%% file: manuscript_standard.tex
\documentclass[9pt,a4paper]{article}

\usepackage{amsmath,amsfonts,mathrsfs, amssymb, amsthm}
\usepackage{calc}
\usepackage[utf8]{inputenc}
\usepackage{comment}
\usepackage{nicefrac}
\usepackage{todonotes}
\usepackage{abstract}
\usepackage{subfig}
\usepackage{rotating}
\usepackage{setspace}
\usepackage[font={small}]{caption}

\usepackage{ulem}
\usepackage{titlesec}
\titleformat{\section}{\large\bfseries}{\thesection}{1em}{}
\titleformat{\subsection}{\normalsize\bfseries}{\thesubsection}{1em}{}

\usepackage[english]{babel}
\usepackage[round]{natbib}

\usepackage{paralist}
\usepackage{fancyhdr,lastpage}

\usepackage{color,hyperref}

\usepackage{tabularx}
\newcolumntype{R}[1]{>{\raggedleft\arraybackslash}p{#1}}

\newcommand{\tab}{\hspace{4 pt}~}

\newcommand{\st}{\text{s.t.} \quad}
\newcommand{\free}{\text{ (free) } }
\newcommand{\botcomma}{\; , \,}

\newcommand{\fixedx}{\mathbf{\overline{x}_i}}
\newcommand{\upfixedx}{^{\mathbf{(\overline{x}_i)}}}
\newcommand{\one}{\mathbf{1}}
\newcommand{\zero}{\mathbf{0}}
\newcommand{\upone}{^{\mathbf{(1)}}}
\newcommand{\upzero}{^{\mathbf{(0)}}}

\newcommand{\on}{^{\,\text{on}}}
\newcommand{\off}{^{\,\text{off}}}

\DeclareMathOperator*{\argmin}{\arg \min}

\newtheorem{theorem}{Theorem}
\newtheorem{definition}{Definition}
\newtheorem{lemma}{Lemma}
\newtheorem{corollary}{Corollary}

\newtheoremstyle{normal}
{10pt}
{10pt}
{\normalfont}
{}
{\bfseries}
{}
{0.8em}
{\bfseries{\thmname{#1}\thmnumber{ #2}.\thmnote{ \hspace{0.5em}(#3)\newline}}}

\theoremstyle{normal}

\newcounter{Assum}
\renewcommand\theAssum{A\arabic{Assum}}
\newcommand{\itemAss}[1]{%
\refstepcounter{Assum}%
\label{assum:#1}%
\item[\bf{\theAssum}]
}

\newcommand{\refAss}[1]{\textbf{\ref{assum:#1}}}

\AtBeginDocument{}

\title{\Large An exact solution method for \\ binary equilibrium problems with compensation \\ 
and the power market uplift problem%
\thanks{The authors would like to thank Ibrahim Abada, Benjamin~F. Hobbs, J.~David Fuller, Daniel Robinson, Carlos Ruiz, Tue Vissing Jensen, Ericson Davis as well as several anonymous reviewers for valuable comments and discussions.}}
\author{\normalsize Daniel Huppmann\textsuperscript{a,b}, Sauleh Siddiqui\textsuperscript{a,c} \\
\small huppmann@iiasa.ac.at, siddiqui@jhu.edu
}
\date{\normalsize Preprint of manuscript published in the \\
\emph{European Journal of Operational Research} \\
\footnotesize(\href{https://dx.doi.org/10.1016/j.ejor.2017.09.032}{DOI: 10.1016/j.ejor.2017.09.032})
}

\begin{document}
\maketitle
\vspace{-1.2 cm}
\begin{abstract}
We propose a novel method to find Nash equilibria in games with binary decision variables by including compensation payments and incentive-compatibility constraints from non-cooperative game theory directly into an optimization framework in lieu of using first order conditions of a linearization, or relaxation of integrality conditions. The reformulation offers a new approach to obtain and interpret dual variables to binary constraints using the benefit or loss from deviation rather than marginal relaxations. The method endogenizes the trade-off between overall (societal) efficiency and compensation payments necessary to align incentives of individual players. We provide existence results and conditions under which this problem can be solved as a mixed-binary linear program. 

We apply the solution approach to a stylized nodal power-market equilibrium problem with binary on-off decisions. This illustrative example shows that our approach yields an exact solution to the binary Nash game with compensation. We compare different implementations of actual market rules within our model, in particular constraints ensuring non-negative profits (no-loss rule) and restrictions on the compensation payments to non-dispatched generators. We discuss the resulting equilibria in terms of overall welfare, efficiency, and allocational equity. 
\end{abstract}

incentive compatibility, electricity market, power market, uplift payments \\[-4 pt]



\noindent\rule{4.8 cm}{0.5 pt}
\begin{footnotesize}
\begin{itemize}
\item[\textsuperscript{a}] Department of Civil Engineering 
\& Center for Systems Science and Engineering, \\ 
The Johns Hopkins University
\item[\textsuperscript{b}] International Institute for Applied Systems Analysis (IIASA) 
\item[\textsuperscript{d}] Department of Applied Mathematics \& Statistics, The Johns Hopkins University
\end{itemize}
\end{footnotesize}

\newpage

\input{mobqe_manuscript}

\begin{footnotesize}
\bibliographystyle{plainnat}
\bibliography{bibliography_dh}

\appendix
\input{mobqe_appendix}

\end{footnotesize}
\end{document}

%% file: mobqe_manuscript.tex
\section{Introduction} \label{sec:introduction}
There are many 
real-world settings where several players interact in a non-cooperative game with binary decisions, such as electricity markets (on-off decision for a power plant), transportation and facility location models \citep{Caunhye_2012_emergency_logistics},  engineering \citep{Rao_1996_engineering}, as well as agriculture and land-use planning \citep{Toth_2011_agriculture}. Modelling Nash equilibria between players which face both binary and continuous decisions is a challenging problem \citep{Scarf_1990_mathprog_economics}.  Economists and game theorists usually apply brute-force methods by exploring all possible combinations and check every solution for deviation incentives of each player. When market-clearing prices to support a pure-strategy Nash equilibrium in the Walrasian sense do not exist, economists suggest to use multi-part pricing \citep{Hotelling_1938_taxation} or deviate from marginal-cost pricing to a ``second-best'' market outcome, such that no player should lose money from participating \citep{Baumol_Bradford_1970_departures_marginal_pricing}. However, a canonical approach to find pure-strategy Nash equilibria in binary games does not exist.

In many large-scale practical applications, exploring the entire solution space is not realistically possible. A common approach in such cases is to linearize the binary decisions; the Nash equilibrium can then be computed by solving the system of first-order optimality conditions, a.k.a.~equilibrium modeling using mixed complementarity problems or variational inequalities, if certain assumptions on convexity of the linearized problem hold. Recent work seeks a trade-off between relaxation of the complementarity (slackness) conditions or the integrality of discrete constraints to obtain stationary points that are presumed to be equilibria of the original problem \citep{Gabriel_2012_NETS,Gabriel_2013_DCMLCP,Fuller_2017}.

In this work, we focus on applications where a relaxation of optimality conditions or continuous relaxation of the binary decision variable (``linearization'') is either not practical or yields incorrect results. Instead, we derive first-order optimality conditions of the continuous variables for \emph{both states of each binary variable} and include those in an overall equilibrium problem simultaneously. Our method then selects the state of the binary variable and corresponding continuous variable which provides the best response for each individual player. 

Due to the nature of a binary game, there are many instances where no set of strategies and no price vector exists that supports a Nash equilibrium in pure strategies; i.e., there is no outcome where the pay-offs to each stakeholder are such that no player has a profitable deviation. This is due to the non-convexity introduced by the binary decision variables and indivisibilities \citep{ONeill_2005_nonconvexities}. 
We introduce the notion of a ``quasi-equilibrium'' to describe situations where no equilibrium exists, but where a market operator or regulator can assign compensation payments in order to obtain an incentive-compatible outcome. These payments align the incentives of individual players with the objectives of the overall system, such as cost minimization or welfare maximization. A regulator may also choose to intervene when an equilibrium exists but its outcome is inferior to the solution that a benevolent planner might achieve. That is, the market operator may seek to minimize the deviation from the system optimum (i.e., all decisions by one planner) caused by the non-cooperative game among a number of decision makers, each seeking to optimize competing objectives. Our solution approach allows to endogenously consider the trade-off between regulatory intervention to improve market efficiency, and the distortions caused by these interventions.

Electricity markets are the real-world application of binary games which have received the most attention in the mathematical optimization literature \citep{ONeill_2013_power_systems_planning,Liu_2013_tacit,Wogrin_2013_math_programming,Liu_Ferris_2013_pricing_electricity,Philpott_2013_challenges,Bjorndal_2008_EJOR_non_convex_prices,Hu_Ralph_2007,Philpott_2006_unit_commitment,ONeill_2005_nonconvexities}. A challenging problem arises from the on-off decision of power plants, which usually incur substantial start-up or shut-down costs and, if operational, face minimum-generation constraints. Because power markets are usually based on marginal-cost, short-term pricing, the  commitment costs (i.e., start-up costs) are not necessarily covered by resulting market prices. 

As a consequence, many electricity systems have rules that generators must be ``made whole'' or have to be ``in the money''; i.e., they receive  ``uplift payments'' to make sure that they do not lose money from participating in the market. This is commonly referred to as a ``no-loss rule''. However, this may not be required from a game-theoretic point of view, and thereby lead to higher-than-necessary compensation payments. At the same time, there might exist regulations that only power plants that are actually generating electricity can receive compensation -- the rationale being that it may create perverse incentives for market participants to be paid to not do something. We will discuss and illustrate in a numerical example how such market rules can actually overly restrict operational efficiency and thereby reduce welfare.

The outline of this paper is as follows: in the next section, we summarize current approaches to solve binary Nash games and place our contribution in the context of methods applied to solve such problems in the power sector. In Section~\ref{sec:model}, we propose an exact solution method to solve binary equilibrium problems. The obtained multi-objective program explicitly incorporates the trade-off between overall efficiency and compensation payments in cases where no equilibrium exists. Section~\ref{sec:example} applies our method to a power market example from the literature to illustrate its advantages and flexibility to incorporate distinct market rules regarding uplift payments. Section~\ref{sec:conclusion} concludes with a discussion on methods, other possible applications, and future work.\footnote{The Appendix provides computational results for a numerical test case using a larger data set than the stylized example in Section~\ref{sec:example}. The GAMS codes for the stylized example, the numerical test case, as well as an additional example for a resource market application with multiple binary investment decisions in production and pipeline capacity for several player are available for download at \url{https://github.com/danielhuppmann/binary_equilibrium} under a \emph{Creative Commons Attribution 4.0 International License}.}

\section{Current approaches to solve binary games} \label{sec:current}
In this section, we motivate our method by describing how current solution methods for binary games obtain equilibria, and we identify where our formulation can improve this process. While there exist brute-force methods \citep{avis2010enumeration,audet2006enumeration,von2002computing} that solve for an equilibrium considering all possible combinations of the binary variables and check ex-post for deviation incentives, we want to concentrate on mathematical programming techniques for obtaining equilibria. For large-scale applications such as those considered in this work, computational efficiency proves a hurdle in these brute-force methods. Solving a large number of equilibrium problems is not very elegant and suffers from a curse of dimensionality, because the number of equilibrium problems to be solved is~$2^k$, where~$k$ is the number of binary variables. Therefore, mathematicians and Operations Researchers are constantly looking for ways to apply advances in Variational Inequalities and Integer Programming to develop faster methods to solve such problems.

\subsection{Optimization and equilibrium modeling}
Game theory and equilibrium problems have been an integral part of the history of mathematical programming.  First-order optimality (Karush-Kuhn-Tucker, KKT) conditions, derived from each individual player's optimization problem, can be solved simultaneously by stacking them to form an equilibrium problem. Interpretations from dual variables to constraints in a game theory analysis provide essential information in equilibrium problems and are often interpreted as prices or marginal benefits for individual players \citep{Facchinei_2003_vi,Ferris_Pang_1997_SIAM,Murphy_1982_math_programming}. 

However, this relationship between optimality conditions and equilibrium problems fails once a game includes binary decision variables. The reason is that optimality conditions cannot be directly derived for binary optimization problems. Thus, applied researchers aim to solve such optimization problems in other ways. 
A method based on a trade-off between relaxing the integrality and the complementarity constraints is developed by \cite{Gabriel_2013_DCMLCP}. While relaxing integrality has been employed as a way to solve integer programs, relaxing complementarity -- essentially the optimality conditions -- was the novel idea of their contribution. 

A similar problem is tackled by \cite{Fuller_2017}; they propose a \emph{minimum disequilibrium} model, defining disequilibrium as the difference between the pay-off in the socially optimal outcome and the individually optimal decision, summed over all players. That is, they seek to minimize the aggregated opportunity costs for all market participants from following the instructions of a social planner. The authors relate the MD model both to the results obtained by a social planner and to the model proposed by \cite{Gabriel_2013_DCMLCP}.

One alternative recent method to tackle binary equilibrium problems focuses on solving integral Nash-Cournot games \citep{Todd_2014_equilibrium_in_integers} and provides an efficient algorithm to obtain equilibria. This method works very well for a specific integer game with no constraints, but the algorithm is not applicable to the broad class of binary-constrained games considered in this paper.

\subsection{Dual variables in binary programs}
As mentioned above, dual variables in constrained convex optimization contain useful information both for computational purposes and interpretation of the problem under consideration. 
However, in mathematical programs with binary or discrete constraints, the interpretation of dual variables as marginal relaxation is not valid because of the non-convex and disjoint feasible region. This is related to the difficulty of determining the value function of the original problem \citep{Guzelsoy_2007_MILP_duality}.
To overcome this caveat and obtain dual variables in such cases, the following approach is often used \citep[cf.][]{ONeill_2005_nonconvexities}. Consider the general constrained problem:
\begin{align} 
\begin{array}{ll}
\displaystyle\min_{x,y} \quad f\big(x,y\big) & \\
~ \st g\big(x,y\big) \leq 0 &
\text{, where } x \in \big\{0,1\big\}^n, ~y \in \mathbb{R}^m
\end{array} \label{prob:originalMIP}
\end{align}

To obtain dual variables to the constraints~$g(x,y)$, this problem is commonly solved in a two-step procedure: first, the original problem~\eqref{prob:originalMIP} is solved using integer programming techniques; then, the binary variables~$x$ are linearized, i.e., the original problem is replaced by the following:
\begin{align} 
\begin{array}{ll}
\displaystyle\min_{x,y} \quad f\big(x,y\big) & \\
~ \st g\big(x,y\big) \leq 0 &
\text{, where } x \in \big[0,1\big]^n \subset \mathbb{R}_+^n, ~y \in \mathbb{R}^m
\end{array} \tag{\ref{prob:originalMIP}\textsuperscript{linear}}
\end{align}

Finally, constraints are added to fix these variables at the level determined to be optimal,~$x^*$, in the first step:
\begin{align}
\min_{x,y} \quad f\big(x,y\big)& \nonumber \\
\st g\big(x,y\big)& \leq 0 \quad \ \, (\lambda) \label{prob:relaxedfixed} \\
x &= x^* \quad (\mu) \nonumber \quad
\text{, where } \big(x,y\big) \in \mathbb{R}^{n+m} \nonumber
\end{align}
Solving the reformulated problem~\eqref{prob:relaxedfixed} allows to interpret the dual variables~$(\lambda,\mu)$ in the sense of multipliers or shadow values; offering these prices as contracts to market participants yields a Nash equilibrium. The dual variables~$\mu$ are not part of the original problem, they are obtained from the linearization and can be thought of as a ``price [\dots] representing the integral activity for (each) agent'' \citep[p.~279]{ONeill_2005_nonconvexities}.

These duals are also important for integer programs, so that most numerical solvers automatically report these values when solving mixed-integer programs. However, one must be careful when using this approach in practical applications, as these duals cannot be readily interpreted as marginal relaxations of the original binary model -- that is, the marginal value~$\lambda$ of the linearized fixed program cannot be interpreted as dual to the constraint of the original, mixed-integer program (problem~\ref{prob:originalMIP}). This is, however, what many power markets are currently doing in practice: they use the dual variable to the energy balance constraint as locational marginal price and clear the market based on these pay-offs. The dual prices of the binary activities~$\mu$ are neglected. Instead, market operators assign compensation payments to make whole individual generators after the fact.

\subsection{Uplifts, compensation, and equilibria in power markets}\label{sec:current:expostcompensation}
There already exists a substantial breadth of Operations Research literature with regard to electricity markets and pricing in non-convex problems, and binary games are a prevalent concern in this area. The current practice in many centrally dispatched power markets is that, first, the welfare-optimal dispatch is computed by the Independent System Operator (ISO) and locational marginal prices (LMP) in the network are determined using the two-step approach outlined above. Compensation to individual players are then calculated after market-clearance to ensure that no market participant incurs financial losses based on these prices. These are often called \emph{uplift}, \emph{make-whole payments} or \emph{bid cost recovery}, though actual implementations and rules differ across markets.

System operators usually have non-confiscatory compensation rules \citep{Sioshansi_2014_central_commitment}. This means that they do not assign penalties for deviation, but only disburse positive compensation payments. In that respect, current market operation deviates from contracts \emph{T} proposed by \cite{ONeill_2005_nonconvexities}, which are derived from all duals~$(\lambda,\mu)$. Instead, standard compensation payments are based on the pay-offs from LMPs (the dual variable or vector $\lambda$ only, in particular the duals to the nodal energy balance constraint). It is important to note that these two are not equivalent.

This approach does not actually guarantee that the incentives of all players are aligned in the resulting market outcome, because the nature of the non-cooperative binary game between market participants is side-stepped. Generators that are not dispatched by the ISO may have an incentive to enter the market, if they earned positive profits given observed market prices, or to deviate from the announced schedule. Some markets allow \emph{self-scheduling}, which gives generators the option to determine their dispatch individually rather than surrendering their generation decision to the ISO \citep[cf.][]{Sioshansi_2010_self-commitment}.

An alternative to the current approach is the minimum uplift or convex hull pricing method, which relies on a convex approximation of the lower bound of the aggregate cost function to derive prices and the minimal uplifts to support the market outcome \citep{Schiro_2015_convexhull,Gribik_2007_price_uplift,Hogan_2003_minimum_uplift}. This method acknowledges that compensation is required to deter generators from following profitable deviations from the dispatch chosen by the ISO. Alas, using the convex hull relaxes the integrality of the underlying problem, and therefore also does not solve for the exact solution to the non-cooperative game between generators.

An important problem of the two-step approach arises from the fact that the budget for necessary compensation payments is not considered when determining the dispatch, but only computed ex-post. This neglects the potential trade-off between efficient market operation and minimizing the budget required for compensation payments, which is usually funded from fees or levies on market participants. These fees may in turn cause distortions in the market.  It is easy to conceive of situations where accepting a slight reduction in market efficiency (i.e., lower welfare, higher costs for dispatch) allows to significantly reduce the compensation payments required. The illustrative example in Section~\ref{sec:example} shows just such a situation.

The method developed in this work tackles these caveats of current approaches and proposes an exact solution method for games in binary variables. Our method offers an important practical advantage: it allows to directly balance efficient market operation based on an exact method for finding solutions to binary equilibrium problems, on the one hand, with the amount of compensation payments to ensure that these outcomes are stable against deviation by individual players, on the other.

\subsection{Marginal relaxation vs.~the loss from a binary deviation}\label{sec:current:switchvalue}
There is a further caveat of using the duals of problem~\eqref{prob:relaxedfixed} for algorithms and (economic) interpretation of results: this approach introduces the dual (vector)~$\mu$ as the marginal relaxation of the constraint that fixes~$x$ at its optimal value. However, it is more appropriate to ask not about a marginal relaxation, but a switch from one possible value of the binary variable to the other. 

We introduce the ``switch value''~$\kappa$ as the benefit or loss incurred by switching from one solution to the binary problem~$f(x^\circ,y^\circ)$ to the optimal value of the objective function given that the binary variable takes the other value,~$x^\times=1-x^\circ$. Here,~$y^\circ$ is chosen so as to minimize $f(x^\circ,y)$, i.e., $y^\circ=\argmin_y f(x^\circ,y)$, and~$y^\times$ is determined equivalently.

Then,~$\kappa$ can be determined by computing:
\begin{align*}
\kappa = - f\big(x^\circ,y^\circ\big) + f\big(x^\times,y^\times\big) .
\end{align*}
If $\kappa$ is strictly positive, switching in the binary variable from $x^\circ$ to $x^\times$ incurs a loss of $\kappa$; hence, $x^\circ$ is the optimal decision. If $\kappa=0$, the objective values are identical and the player is indifferent between the two options.

When~$x \in \{0,1\}^n$ is a binary vector rather than a one-dimensional variable, the switch value can be computed by comparing the objective value for a possible realization $x^\circ$ to the outcome for all other permutations $\mathfrak{S}(\{0,1\}^n)$ of the binary vector and choosing the most beneficial (minimal) alternative:
\begin{align*}
\kappa(x^\circ) = - f\big(x^\circ,y^\circ\big) 
+ \min_{x^\times \in \mathfrak{S}(\{0,1\}^n) \setminus x^\circ} f\big(x^\times,y^\times\big) 
\end{align*}
As before, if $\kappa$ is strictly positive, this implies that $x^\circ$ is optimal, and $\kappa=0$ means that there is (at least) one alternative in the binary decision vector with the same objective value.

This formulation still requires comparing the objective values of $2^n$~alternatives and solving for the optimal level of the continuous variables~$y$ in each case. Hence, this approach may not seem like an improvement. The big advantage will become apparent in settings where multiple players~$i \in I = \left\{1,\dots,p\right\}$ interact and one solves for an equilibrium between them. A brute-force approach would require to solve all permutation across players and their options in binary variables ($2^{pn}$). Building on the approach identified above, this can be transformed to a multi-objective optimization problem with $p\,2^n$ options. We will discuss the analytical properties in subsequent sections and present a numerical analysis using a larger-scale dataset in the Appendix.

In the method proposed below, we use this notion of a switch value~$\kappa$ to choose between equilibria in such games with binary decisions. This variable also serves as a selection mechanism in such cases where no binary equilibrium exists; it can then be used as a solution strategy to find an appropriate quasi-equilibrium. This approach holds promise with regard to algorithmic advances of binary and integer programming, as well as allow a better representation of real-world problems in economics, engineering, and beyond. 

\section{An exact solution for binary equilibrium problems} \label{sec:model}
We now turn to our exact solution method to solve an equilibrium problem with binary variables. 
The core idea for our approach is as follows: for each player, we derive the first-order optimality conditions with respect to the continuous decision variables for each state of the binary variable. In addition, we formulate an explicit incentive-compatibility constraint to ensure that each player chooses the state of the binary variable that is most beneficial to her. 

For ease of notation and formulating a concise and simple exposition of our approach, we drop the index on the binary variable and describe the method in the case where each player has exactly one binary decision variable, while the number of continuous decision variables and constraints is arbitrary. Nevertheless, the approach works for any problem with a finite number of binary decision variables. To illustrate this feature, the electricity market example presented in the following section has multiple binary decision variables per player.

The game is defined by a set of players $i \in I = \left\{1,\dots,p\right\}$, where each player seeks to minimize an objective function $f_i(\cdot)$. In the following formulation, each player has a (vector of) continuous decision variable(s) $y_i \in \mathbb{R}^m$, binary decision variable $x_i \in \{0,1\} $ and a set of $k$~constraints~$g_i:\mathbb{R}^{m} \times \{0,1\}\rightarrow \mathbb{R}^k$ with a vector of length~$k$ of associated dual variables~$\lambda_i$. As elaborated in the previous section, these dual variables are only meaningful for a fixed~$x_i$. The feasible region of each player is denoted by~$\mathcal{K}_i=\big\{(x_i, y_i)\;|\;g_i(x_i,y_i) \leq 0\big\}$.

Each player's optimization problem reads as follows:
\begin{subequations} \label{model:individual_optimization_problem}
	\begin{align}
	\min_{x_i \in \{0,1\},y_i \in \mathbb{R}^m}\quad &f_i\big(x_i,y_i,y_{-i}(x_{-i})\big) \\
	\st & \ g_i\big(x_i,y_i\big) \leq 0 \quad (\lambda_i) 
	\end{align}
\end{subequations}
The vector~\smash{$y_{-i}=(y_{j})_{{j} \in I \setminus \{i\}}$} is the collection of all rivals' decisions in continuous variables, and thus is of dimension~$m \times (p-1)$. The set of feasible strategies by the rivals is~\smash{$\mathcal{K}_{-i}=\prod_{j \in I \setminus \{i\}} \big(y_{j}(x_{j})\big)$}. Because the continuous variables of the rivals' depend on their binary decisions,~$\mathcal{K}_{-i}$ is usually pairwise disjoint and non-convex. The formulation implicitly assumes that each player's pay-off is only affected by the continuous decision variables of her rivals, but not directly affected by their binary variables. This is a simplification only for notational convenience and can easily be relaxed.

A Nash equilibrium to this game is a set of strategies such that each player chooses an optimal strategy given the action by the rivals. This is equivalent to the notion that no player has an incentive to unilaterally change her decision upon observing the decisions of the rivals; there exists no profitable deviation. This is formally defined below; we distinguish between deviation incentives in the binary and the continuous variables to facilitate the exposition.

\begin{spacing}{1.1}
	\begin{definition}[Nash equilibrium in a binary game]\label{definition:BNE}
		We define the binary game as a set of players~$i \in I$, each seeking to solve an optimization problem as given by problem~\eqref{model:individual_optimization_problem}.  A Nash equilibrium to this game is a vector~$\big( (x_i^*, y_i^*) \in \mathcal{K}_i \big)_{i \in I}$ such that~$y_i^*$ is the optimal decision (i.e., best response) by player~$i$ given~$x_i^*$ and~$y_{-i}^*(x_{-i}^*)$,
		\begin{align}
		f_i\big(x_i^*,y_i^*,y_{-i}^*(x_{-i}^*)\big) \leq f_i\big((x_i^*,y_i,y_{-i}^*(x_{-i}^*)\big) \ \forall~y_i \in \big\{y_i\;|\;g_i\big(x_i^*,y_i\big) \leq 0  \big\} \ \forall~i \in I
		\end{align}
		and such that there is no profitable deviation with regard to the binary variable,
		\begin{align}
		f_i\big(x_i^*,y_i^*,y_{-i}^*(x_{-i}^*)\big) \leq f_i\big(x_i^\times,y_i^\times,y_{-i}^*(x_{-i}^*)\big) \quad \forall~i \in I,
		\end{align}
		where~$x_i^\times$ is the alternative value of~$x_i$, i.e.,~$x_i^\times=1-x_i^*$, and~$y_i^\times$ is a best response of player~$i$ under the assumption that $x_i=x_i^\times$, i.e.,
		\begin{align}
		f_i\big((x_i^\times,y_i^\times,y_{-i}^*(x_{-i}^*)\big) \leq &~ f_i\big((x_i^\times,y_i,y_{-i}^*(x_{-i}^*)\big)  \nonumber \\
		& \quad \forall~y_i \in \Big\{y_i\;|\;g_i\big(x_i^\times,y_i\big) \leq 0 \Big\} \quad \forall~i \in I.
		\end{align}
	\end{definition}
\end{spacing}

Because existence or uniqueness of equilibria cannot be guaranteed in binary games, we need to devise a method to select among several outcomes, or to arrive at a desired point which is ``almost'' an equilibrium. For this purpose, we introduce a \emph{market operator}, as a coordination agent and equilibrium selection mechanism. This entity is modeled as the upper-level player within a hierarchical, two-stage setup, where the lower-level constraints represent the binary equilibrium problem. She guides the players towards a desirable outcome and assigns compensation payments if necessary.

We formally introduce the term \emph{quasi-equilibrium} for solutions to the binary game that are not Nash equilibria according to Definition~\ref{definition:BNE}, but where incentive-compatibility can be ensured with appropriate compensation payments.

\begin{spacing}{1.1}
	\begin{definition}[Quasi-equilibrium in a binary game with compensation]\label{definition:BNE:compensation}
		\hfill \newline
		We define the binary game with compensation as a set of players~$i \in I$, each seeking to solve an optimization problem as given by problem~\eqref{model:individual_optimization_problem}.  A binary quasi-equilibrium to this game is a vector~$\big( (x_i^*, y_i^*) \in \mathcal{K}_i \big)_{i \in I}$ and a compensation vector $\big(\zeta_i\in \mathbb{R}_+ \big)_{i \in I}$ such that for each player:
		\begin{enumerate}
			\item $y_i^*$ is the optimal feasible decision (i.e., best response) by player~$i$ given~$x_i^*$ and~$y_{-i}^*(x_{-i}^*)$,
			\begin{align}
			f_i\big(x_i^*,y_i^*,y_{-i}^*(x_{-i}^*)\big) \leq &~ f_i\big(x_i^*,y_i,y_{-i}^*(x_{-i}^*)\big) \nonumber \\ 
			&\quad \forall~y_i \in \Big\{y_i\;|\;g_i\big(x_i^*,y_i\big) \leq 0  \Big\} \quad \forall~i \in I,
			\end{align}
			\item no player can improve her own pay-off by deviating from $x_i^*$ by more than the compensation payment~$\zeta_i$; i.e., the compensation is at least as great as the benefit from deviation with regard to the binary variable. Hence, there is no profitable deviation with regard to the binary variable given the compensation payment,
			\begin{align}
			f_i\big(x_i^*,y_i^*,y_{-i}^*(x_{-i}^*)\big) - \zeta_i \leq f_i\big(x_i^\times,y_i^\times,y_{-i}^*(x_{-i}^*)\big) \quad \forall~i \in I \label{definition:BNE:incentive:compatibility}
			\end{align}
			where~$x_i^\times$ and~$y_i^\times$ are defined as in Definition~\ref{definition:BNE},
			\item and the compensation payments are minimal, i.e., if a compensation payment is required for a player, then the incentive-compatibility condition~\eqref{definition:BNE:incentive:compatibility} holds with equality. That is,
			\begin{align}
			\zeta_i = ~\min_{\widetilde{\zeta_i} \in \mathbb{R}_+} \quad &\widetilde{\zeta_i}  \nonumber \\ 
			\emph{\st} & f_i\big(x_i^*,y_i^*,y_{-i}^*(x_{-i}^*)\big) - \widetilde{\zeta_i} \leq f_i\big(x_i^\times,y_i^\times,y_{-i}^*(x_{-i}^*)\big) \quad \forall~i \in I.
			\end{align}
		\end{enumerate}
	\end{definition}
\end{spacing}
Note that when $\sum\nolimits_{i \in I} \zeta_i =0$, the binary quasi-equilibrium is also a Nash equilibrium in a game without compensation. In the definition of the quasi-equilibrium, we directly incorporate the notion that the compensation payments should be minimal. This is helpful because it eliminates those incentive-compatible solutions where the market operator ``over-compensates'' some players, and it allows to focus on a smaller set of candidate solutions.
\footnote{For games where the individual players' optimization problems are non-convex with continuous variables, \cite{Pang_2013_quasi_Nash} introduce the notion of a ``quasi-Nash equilibrium'' to describe solutions that are stationary points derived from relaxed constraint qualifications. 
In the definition used here, we are looking at a distinct concept of an equilibrium. 
}

\subsection{Determining each player's best response}
In the definitions above, we have simply stated that the continuous decision variables~$y_i^*$ are optimal for player~$i$ given the binary variable and the rivals' decisions. In order to efficiently compute this best response of each player, we use first-order optimality conditions with regard to the continuous decision variables. Hence, we need to make sure these conditions are necessary and sufficient so that we can capture the entire equilibrium set. An assumption on compactness is also needed for the selection of certain parameters of our method.
\begin{itemize}
	\itemAss{KKT:conditions} Assume that for each player~$i \in I$, problem~\eqref{model:individual_optimization_problem} is such that the first-order optimality (KKT) conditions are necessary and sufficient with respect to the variables~$y_i$, and the feasible region defined by the constraints $g_i\big(x_i,y_i\big)$ is compact and non-empty, for a fixed realization of~$x_i$ and for any fixed feasible strategy by the rivals~$y_{-i} \in \mathcal{K}_{-i}$
	.
\end{itemize}
As an example, the KKT conditions are necessary and sufficient for problem~\eqref{model:individual_optimization_problem} if~$f_i(\fixedx,\cdot,y_{-i})$ are convex and~$g_i(\fixedx,\cdot )$ affine for any fixed value~$\fixedx \in \{0,1\}$ and any fixed vector~$y_{-i} \in \mathcal{K}_{-i}$.

Let the vector~$(x_i^*,y_i^*)$ denote the best response for each player within the overall problem, given the decision vector $\big(y_{-i}(x_{-i})\big)_{i\in I}$ by all rivals, and let~\smash{$\widetilde{y}_i\upfixedx$} denote the best response of player~$i$ for a fixed~$x_i=\fixedx \in \{0,1\}$. Then, the objective value~\smash{$f_i\big(\fixedx,\widetilde{y}_i\upfixedx,y_{-i}(x_{-i})\big)$} is the best pay-off that a player can do given~$\fixedx$ and the rivals' strategies.

Under Assumption~\refAss{KKT:conditions}, if the value of $x_i$ is fixed at $\fixedx$, the best response $\widetilde{y}_i\upfixedx$ can be found by solving the respective first-order optimality conditions:
\begin{subequations} \label{model:FOC}
	\begin{align}
	0 = \nabla_{y_i} \,  f_i\big(\fixedx,\widetilde{y}_i\upfixedx,y_{-i}(x_{-i})\big) + \widetilde{\lambda}_i\upfixedx \, \nabla_{y_i} \, g_i\big(\fixedx,\widetilde{y}_i\upfixedx\big)
	\quad &\botcomma \quad \widetilde{y}_i\upfixedx \free \\
	0 \geq g_i\big(\fixedx,\widetilde{y}_i\upfixedx\big) \quad &\bot \quad \widetilde{\lambda}_i\upfixedx \geq 0
	\end{align}
\end{subequations}

Player $i$ will choose the binary variable $x_i$ such that its objective value is minimal given the decisions of the rivals $y_{-i}(x_{-i})$. Mathematically, the best response of player~$i$ regarding her binary variable~$x_i$ can be written as follows:
\begin{subequations} \label{model:incentive:compatibility} 
	\begin{align}
	f_i\big(\one,\widetilde{y}_i\upone,y_{-i}(x_{-i})\big) < f_i\big(\zero,\widetilde{y}_i\upzero,y_{-i}(x_{-i})\big) \quad & \Rightarrow \quad x_i^* = 1 \\
	f_i\big(\one,\widetilde{y}_i\upone,y_{-i}(x_{-i})\big) > f_i\big(\zero,\widetilde{y}_i\upzero,y_{-i}(x_{-i})\big) \quad & \Rightarrow \quad x_i^* = 0 \\
	f_i\big(\one,\widetilde{y}_i\upone,y_{-i}(x_{-i})\big) = f_i\big(\zero,\widetilde{y}_i\upzero,y_{-i}(x_{-i})\big) \quad & \Rightarrow \quad  x_i^* = \{0,1\}
	\end{align}
\end{subequations}
The logic of conditions~\eqref{model:incentive:compatibility} is similar to the notion of incentive compatibility in game theory, i.e., there exists no profitable deviation given the decisions of all rivals. 
Hence, a vector~$\big(x_i^*,y_i^*(x_i^*)\big)_{i \in I}$ that satisfies the incentive-compatibility constraints in Definition~\ref{definition:BNE} for each player constitutes a Nash equilibrium. If the incentive-compatibility condition is not satisfied for any feasible strategy, it may be necessary to financially compensate a player to ensure that she doesn't deviate, as stated in  Definition~\ref{definition:BNE:compensation}. 

A direct implementation of the implicit ``if-then''-logic requires additional binary variables and thereby considerably increases numerical complexity. We overcome this drawback by proposing a mathematically equivalent formulation 
using the original binary variables of the players. The resulting overall program will be shown in problem~\eqref{model:overall:problem}; but first, we will discuss the reformulation and introduce the equilibrium selection mechanism in more detail.

\subsection{An efficient formulation of incentive compatibility}
We introduce four non-negative variables $\big(\kappa_i\upone,\kappa_i\upzero,\zeta_i\upone,\zeta_i\upzero\big)$ for each player, and a sufficiently large scalar (or vector of scalars)~\smash{$\widetilde{K}$}. 
The vector~$\kappa_i\upfixedx$ is the switch value introduced in Section~\ref{sec:current:switchvalue}; it can be interpreted as the loss the player would incur by switching from her optimal value of the binary variable to the alternative.
The vector~$\zeta_i\upfixedx$ denotes compensation payments to guarantee incentive compatibility; in cases where the market operator requires a player to act against her own objectives, this payment ensures that the player does not have a profitable deviation.

The scalar $\widetilde{K}$ must be large enough so that it does not inadvertently constrain the variables $\big(\kappa_i\upone,\kappa_i\upzero,\zeta_i\upone,\zeta_i\upzero\big)$. Since these are differences in objective function values, this implies that $\widetilde{K}$ must be larger than the size of the range of $f_i(\cdot)$. By Assumption~\refAss{KKT:conditions}, each $f_i(\cdot)$ is continuous over a compact feasible region, thus achieves both its maximum and minimum within the feasible region. An efficient technique to choose $\widetilde{K}$ is to linearize the binary variables in the individual optimization problems and minimize and maximize over $f_i(\cdot)$ to find the largest difference possible. Note that the role of $\widetilde{K}$ here and in the subsequent sections is to enforce the disjunction between two choices of the binary variable $x_i$. It is the \emph{disjunctive constraints} formulation introduced by \cite{Fortuny-Amat_1981}.

The vectors~$\kappa_i\upfixedx$ and~$\zeta_i\upfixedx$ are not dual variables in the original sense, but they do contain similar information regarding the solution. Hence, they are analogous in interpretation to a dual -- but in terms of a binary deviation, not in the sense of a marginal relaxation. Alas, the term ``shadow price'' often used in economics as synonymous for dual variables could also be applied here.

We can now replace the incentive compatibility conditions (equations~\ref{model:incentive:compatibility}) by a more efficient formulation:
\begin{subequations} \label{model:incentive:compatibility:reformulated} 
	\begin{align}
	f_i\big(\one,\widetilde{y}_i\upone,y_{-i}\big) + \kappa_i\upone - \zeta_i\upone - \kappa_i\upzero + \zeta_i\upzero &= f_i\big(\zero,\widetilde{y}_i\upzero,y_{-i}\big) \label{model:incentive:compatibility:reformulated:1} \\
	\kappa_i\upone + \zeta_i\upone&\leq  x_i \, \widetilde{K} \\[2 pt]
	\kappa_i\upzero + \zeta_i\upzero&\leq \big(1-x_i\big) \, \widetilde{K} \\[2 pt]
	\kappa_i\upone, \kappa_i\upzero, \zeta_i\upone, \zeta_i\upzero &\in \mathbb{R}_+ \nonumber
	\end{align}
\end{subequations}

The market operator selects a solution such that the first-order optimality conditions and the incentive-compatibility constraints are satisfied for all players. In line with the definition of the quasi-equilibrium as the minimal compensation payment for each player~$i$, the variables~$\kappa_i\upfixedx$ and~$\zeta_i\upfixedx$ cannot both be strictly greater than zero at a solution; this will be shown after we formally introduce the market operator.

\begin{table}[b]
	\begin{center}
		\begin{small}
			\begin{tabular}{c|cc|cccc|c} \label{table:equilibrium:cases}
				& individually  &  equilibrium $x_i$ chosen & & & & & incentives \\
				case & optimal $x_i$ &  by market operator & $\kappa_i\upone$&$\kappa_i\upzero$&$\zeta_i\upone$&$\zeta_i\upzero$ & aligned \\ 
				\hline \hline 
				I & $\one$ & $\one$ & $>0$ & $0$ & $0$ & $0$ & yes \\
				II & $\zero$ & $\zero$ & $0$ & $>0$ & $0$ & $0$ & yes \\
				III & $\zero$ & $\one$ & $0$ & $0$ & $>0$ & $0$ & no \\
				IV & $\one$ & $\zero$ & $0$ & $0$ & $0$ & $>0$ &  no \\
				V  & indifferent & $\one$ / $\zero$ & $0$ & $0$ & $0$ & $0$ &  yes
			\end{tabular}
		\end{small}
	\end{center}
	\caption{Incentive alignment between a player's individually optimal decision (her best response) and the quasi-equilibrium chosen by the market operator} \label{table:incentive_alignment}
\end{table}

Let us illustrate and discuss the interpretation of the variables~$\kappa_i\upfixedx$ and~\smash{$\zeta_i\upfixedx$} in more detail.
The question is whether the solution for the overall equilibrium problem chosen by the market operator is aligned with the best response of each player. By this, we mean whether a player's individually optimal decision coincides with the quasi-equilibrium chosen by the market operator.
There are five possible outcomes regarding the incentive alignment of an individual player and the market operator; the cases are illustrated in Table~\ref{table:incentive_alignment}. In cases~I and~II, the incentives are aligned, as the player would incur losses (a strictly worse pay-off) by deviating from the outcome decided by the market operator. The respective switch value~$\kappa_i\upfixedx$ is strictly positive. 
In cases~III and~IV, the solution chosen by the market operator is not in line with the player's individual best response; only by disbursing compensation payments can the market operator convince the player not to deviate to the individually optimal decision. As a consequence, the respective compensation payment~$\zeta_i\upfixedx$ is strictly positive, and a quasi-equilibrium is realized. 
In the last case (no.~V), the player is indifferent between her options, so the market operator is not restricted in selecting either outcome. The player does not have a positive switch value in either direction, and no compensation is required.

\subsection{Translating each player's best response into the overall game}
From equations \eqref{model:FOC}, we have obtained two optimal decision vectors,~$\widetilde{y}_i\upfixedx$, for each player for both values that the variable~$\fixedx$ can take. We now need to translate which of these two decision variables is realized in the quasi-equilibrium and ``seen'' by the rivals in their own optimization problem:
\begin{subequations} \label{model:translation:y}
	\begin{align}
	\widetilde{y}_i\upzero - x_i \, \widetilde{K} \leq 
	y_i &\leq \widetilde{y}_i\upzero + x_i \, \widetilde{K} \\
	\widetilde{y}_i\upone - \big( 1 - x_i \big) \, \widetilde{K} \leq
	y_i &\leq \widetilde{y}_i\upone + \big( 1 - x_i \big) \, \widetilde{K}
	\end{align}
\end{subequations}
The logic of constraints \eqref{model:translation:y} is straightforward: the decision vector~$y_i$, as it is considered by the rivals and the market operator in their optimization problems, must be equal to the optimal decision~$\widetilde{y}_i\upfixedx$ for whichever value of $x_i$ is the solution in the quasi-equilibrium, 
i.e., $x_i=0\Rightarrow y_i = \widetilde{y}_i\upzero$ and $x_i=1\Rightarrow y_i = \widetilde{y}_i\upone$. The parameter~$\widetilde{K}$ must be chosen suitably large so as not to constrain the continuous decision variable(s). This implies that $\widetilde{K}$ must be larger than the size of the domain of $y_i$. As argued before, each $y_i$ is continuous and, by Assumption~\refAss{KKT:conditions}, over a compact feasible region. A suitable value for $\widetilde{K}$ can be found by linearizing the binary variables in the individual optimization problems and minimize and maximize over $y_i$ to find the largest difference possible. 

We can now combine the incentive-compatibility constraints~\eqref{model:incentive:compatibility:reformulated} with the equilibrium conditions~\eqref{model:FOC} for the continuous decision variables into one set of constraints. The non-linearity of the complementarity conditions \eqref{model:FOC} can be readily reformulated applying disjunctive constraints \citep{Fortuny-Amat_1981} or using SOS type~1 variables \citep{Siddiqui_2013_SOS1}. 

\subsection{A multi-objective program subject to binary quasi-equilibria}\label{sec:model:mopbqe}
So far, we have only replaced a number of equilibrium problems (for each possible combination of binary variables) by a set of integer constraints that exactly represent all binary quasi-equilibria. Next, we can apply multi-objective programming to direct the game towards desired solutions. To this end, we introduce the market operator, and we assume that she seeks to minimize an objective function consisting of two terms: a function~$F(\cdot)$, which only depends on the actual market outcome (efficiency of the solution; cost-minimization or welfare-maximization may be a natural choice for this term) and a function~$G(\cdot)$, which serves as a regularizer. Parameterizing this function appropriately allows to weight between the different terms; in economic applications, it can be interpreted as a penalty term that seeks to minimize the compensation payments required to ensure incentive compatibility of the market solution.

The overall problem is a mathematical program subject to a binary equilibrium problem, where $\fixedx= \{0,1\}$ are the binary variables in the lower-level problem.
\begin{subequations} \label{model:overall:problem}
	\begin{align}
	\min_{
		\substack{x_i,y_i,
			\widetilde{y}\upfixedx_i,\widetilde{\lambda}\upfixedx_i \\  \kappa\upfixedx_i,\zeta\upfixedx_i}}
	\quad F\Big(\big(x_i,y_i\big)_{i \in I}\Big) + G\Big(\big(\zeta\upfixedx_i\big)_{i \in I}\Big) \hspace{4 cm} \label{model:overall:problem:objective}
	\end{align}
	\vspace{-0.5 cm}\begin{align}
	%
	\st \nabla_{y_i} \, f_i\Big(\one,\widetilde{y}_i\upone,y_{-i}\Big) + \big(\widetilde{\lambda}_i\upone\big)^T \nabla_{y_i} \, g_i\Big(\one,\widetilde{y}_i\upone\Big) &= 0 \label{model:overall:problem:KKT:1}\\
	0 \leq - g_i\Big(\one,\widetilde{y}_i\upone \Big)& \ \bot \ \widetilde{\lambda}_i\upone \geq 0 \label{model:overall:problem:KKT:constraints:1} \\
	\nabla_{y_i} \, f_i\Big(\zero,\widetilde{y}_i\upzero,y_{-i}\Big) + \big(\widetilde{\lambda}_i\upzero\big)^T \nabla_{y_i} \, g_i\Big(\zero,\widetilde{y}_i\upzero\Big) &= 0 \label{model:overall:problem:KKT:0}\\
	0 \leq - g_i\Big(\zero,\widetilde{y}_i\upzero \Big)& \ \bot \ \widetilde{\lambda}_i\upzero \geq 0 \label{model:overall:problem:KKT:constraints:0} \\
	f_i\Big(\one,y_i\upone,y_{-i}\Big) + \kappa_i\upone - \zeta_i\upone - \kappa_i\upzero + \zeta_i\upzero&= f_i\Big(\zero,y_i\upzero,y_{-i}\Big) \label{model:overall:problem:incentive} \\
	\kappa_i\upone + \zeta_i\upone&\leq  x_i \, \widetilde{K}  \label{model:overall:problem:duals1} \\
	\kappa_i\upzero + \zeta_i\upzero&\leq \big(1-x_i\big) \, \widetilde{K}  \label{model:overall:problem:duals0} \\
	\widetilde{y}_i\upzero - x_i \, \widetilde{K} \leq y_i &\leq \widetilde{y}_i\upzero + x_i \, \widetilde{K}  \label{model:overall:problem:translate0} \\
	\widetilde{y}_i\upone - \big(1-x_i \big) \, \widetilde{K} \leq
	y_i &\leq \widetilde{y}_i\upone + \big(1-x_i \big) \, \widetilde{K} \label{model:overall:problem:translate1} \\
	x_i \in \{0,1\}, \big(y_i,\widetilde{y}_i\upfixedx \big) \in \mathbb{R}^{3m}, \big(\lambda_i\upfixedx, 
	& \kappa_i\upfixedx,\zeta_i\upfixedx\big) \in \mathbb{R}_+^{2k+4}
	\nonumber
	\end{align}
\end{subequations}
It is important to note that the binary variable of each player has an additional role in this formulation: it also controls which of the potential states with regard to the continuous variables are active and ``visible'' to the rivals (constraints~\ref{model:overall:problem:translate0} and~\ref{model:overall:problem:translate1}). Furthermore, it ensures that the correct switch values and compensation payments are active (constraints~\ref{model:overall:problem:duals1} and~\ref{model:overall:problem:duals0}), in line with Table~\ref{table:incentive_alignment}.

\begin{spacing}{1.1}
	\begin{theorem}[Exact solutions of the binary Nash game]\label{theorem:BNE}
		Under Assumption~\refAss{KKT:conditions}, any vector~$\big(x_i,y_i\big)_{i \in I}$ is a solution to the binary game in Definition~\ref{definition:BNE} if and only if there exists a vector~$\big(\widetilde{y}\upfixedx_i,\widetilde{\lambda}\upfixedx_i,\kappa\upfixedx_i\big)_{i \in I}$, with~$\zeta\upfixedx_i=0 \ \forall ~ i \in I$, such that $\big(x_i,y_i,\widetilde{y}\upfixedx_i,\widetilde{\lambda}\upfixedx_i,\kappa\upfixedx_i\big)_{i \in I}$ is a feasible point to problem~\eqref{model:overall:problem}. 
	\end{theorem}
	\begin{proof}
		\begin{footnotesize}
			First, assume $\big(x_i,y_i,\widetilde{y}\upfixedx_i,\widetilde{\lambda}\upfixedx_i,\kappa\upfixedx_i\big)_{i \in I}$ with~$\zeta\upfixedx_i=0 \ \forall~i\in I$ is a feasible point to problem~\eqref{model:overall:problem}. Then, by Assumption~\refAss{KKT:conditions}, we know that the point~$\big(x_i,y_i\big)_{i \in I}$ is an optimal solution for each player given fixed values of~$x_i$ and~$y_{-i}$ $\forall~i\in I$. This satisfies the first part of Definition~\ref{definition:BNE}. Furthermore, we know that~$\zeta\upfixedx_i=0 \ \forall~i\in I$, and~$\big(\kappa\upfixedx_i\big)_{i\in I}$ will be selected according to the constraints of problem~\eqref{model:overall:problem}. By these constraints, we know that~$f_i\big(x_i,y_i,y_{-i}(x_{-i})\big) \leq f_i\big(x_i^\times,y_i^\times,y_{-i}(x_{-i})\big) \ \forall~i \in I$,
			where~$x_i^\times$ is the alternative value of~$x_i$ (i.e.,~$x_i^\times=1-x_i$) and~$y_i^\times$ is a best response of player~$i$, i.e.,
			\begin{align*}
			f_i\big(x_i^\times,y_i^\times,y_{-i}(x_{-i})\big) \leq  f_i\big(x_i^\times,y_i,y_{-i}(x_{-i})\big) \quad \forall~y_i \in \Big\{y_i\;|\;g_i\big(x_i^\times,y_i\big) \leq 0  \Big\} \quad \forall~i \in I. 
			\end{align*}
			This satisfies the second part of Definition \ref{definition:BNE} and thus we have shown that $\big(x_i,y_i\big)_{i \in I}$  is a solution to the binary game defined by Definition~\ref{definition:BNE}.
			
			Now, we assume that $\big(x_i,y_i\big)_{i \in I}$ is a solution to the binary game defined by Definition~\ref{definition:BNE}. Choose~$\widetilde{K}$ large enough so that it is greater than the difference between any upper and lower bounds on~$y_i \ \forall~i \in I$ and greater than the difference between the minimum and maximum value of~$f_i(\cdot) \ \forall~i \in I$. Such a value exists since by Assumption~\refAss{KKT:conditions}; the feasible set is compact and~$f_i(\cdot)$ is continuous, so the maximum and minimum must exist. Then, by Definition~\ref{definition:BNE}, for any fixed value of $x_i$ and $y_{-i}$, $y_i$ is an optimal solution to the individual player's optimization problem. Thus, you can find $(\widetilde{y}\upfixedx_i,\widetilde{\lambda}\upfixedx_i)$ such that $\big(x_i,y_i,\widetilde{y}\upfixedx_i,\widetilde{\lambda}\upfixedx_i\big)_{i \in I}$ satisfies the constraints~(\ref{model:overall:problem:KKT:1}--\ref{model:overall:problem:KKT:constraints:0}). Take $\zeta\upfixedx_i=0 \ \forall~i\in I$, and choose $\kappa\upfixedx_i$ according to constraint~\eqref{model:overall:problem:incentive}. Thus, for any solution to the binary game in Definition~\ref{definition:BNE} given by~$\big(x_i,y_i\big)_{i \in I}$, we have shown that there there exists a vector~$\big(\widetilde{y}\upfixedx_i,\widetilde{\lambda}\upfixedx_i,\kappa\upfixedx_i\big)_{i \in I}$, with~$\zeta\upfixedx_i=0 \ \forall~ i\in I$, such that $\big(x_i,y_i,\widetilde{y}\upfixedx_i,\widetilde{\lambda}\upfixedx_i,\kappa\upfixedx_i\big)_{i \in I}$ is a feasible point to problem~\eqref{model:overall:problem}. 
		\end{footnotesize}
	\end{proof}
\end{spacing}

The next theorem shows that the method can also be applied to obtain a quasi-equilibrium. Note that we need an assumption on the objective function before we can prove that our method can obtain a quasi-equilibrium.

\begin{itemize}
	\itemAss{upper:level:convex:quadratic} Assume that~$F\big(\cdot\big)$ and~$G\big(\cdot\big)$ are convex quadratic or linear functions for every fixed binary variable~$x_i$ and that $\partial \, G\big(\cdot\big)/\partial \, \zeta_i>0 \quad \forall~i\in I$.
\end{itemize}

\begin{spacing}{1.1}
	\begin{theorem}[Exact solutions of the binary Nash game with compensation]\label{theorem:BNE:compensation}
		Under Assumptions~\refAss{KKT:conditions} and~\refAss{upper:level:convex:quadratic}, if there exists a vector~ $\big(x_i,y_i,\widetilde{y}\upfixedx_i,\widetilde{\lambda}\upfixedx_i,\kappa\upfixedx_i,\zeta\upfixedx_i\big)_{i \in I}$ that is an optimal solution to problem~\eqref{model:overall:problem}, then the vector~$\big(x_i,y_i\big)_{i \in I}$ with compensation $\big(\zeta_i\big)_{i \in I}$ is a solution to the binary game as stated in Definition~\ref{definition:BNE:compensation}. Following the term introduced in Definition~\ref{definition:BNE:compensation}, we refer to this as a \emph{binary quasi-equilibrium}.
		
		Furthermore, under Assumptions~\refAss{KKT:conditions} and~\refAss{upper:level:convex:quadratic}, if~$(x_i,y_i)_{i \in I}$ is a solution to the binary game with compensation $\big(\zeta_i\big)_{i \in I}$ as stated in Definition~\ref{definition:BNE:compensation}, then there exists a vector~\smash{$\big(\widetilde{y}\upfixedx_i,\widetilde{\lambda}\upfixedx_i,\kappa\upfixedx_i,\zeta\upfixedx_i\big)_{i \in I}$}, such that \smash{$\big(x_i,y_i,\widetilde{y}\upfixedx_i,\widetilde{\lambda}\upfixedx_i,\kappa\upfixedx_i,\zeta\upfixedx_i\big)_{i \in I}$} is a feasible point to problem~\eqref{model:overall:problem}. 
	\end{theorem}
	\begin{proof}
		\begin{footnotesize}
			First, assume $\big(x_i,y_i,\widetilde{y}\upfixedx_i,\widetilde{\lambda}\upfixedx_i,\kappa\upfixedx_i,\zeta\upfixedx_i\big)_{i \in I}$ is an optimal solution to problem~\eqref{model:overall:problem}. Then, by~\refAss{KKT:conditions}, we know that the point~$\big(x_i,y_i\big)_{i \in I}$ is an optimal solution for each player given fixed values of $x_i$ and $y_{-i}$. This satisfies the first part of Definition~\ref{definition:BNE:compensation}. Furthermore, we know that  $\kappa\upfixedx_i,\zeta\upfixedx_i \ \forall~i\in I$ will be selected according to the constraints of problem~\eqref{model:overall:problem}. By these constraints, we know that $f_i\big(x_i,y_i,y_{-i}(x_{-i})\big) - \zeta_i \leq f_i\big(x_i^\times,y_i^\times,y_{-i}(x_{-i})\big) \ \forall~i \in I$, where~$x_i^\times$ is the alternative value of~$x_i$ (i.e.,~$x_i^\times=1-x_i$) and~$y_i$ is a best response of player $i$ with fixed~$x_i^\times$, i.e.,
			\begin{align*}
			f_i\big(x_i^\times,y_i^\times,y_{-i}(x_{-i})\big) \leq  f_i\big(x_i^\times,y_i,y_{-i}(x_{-i})\big) 
			\quad \forall~y_i \in \Big\{y_i\;|\;g_i\big(x_i^\times,y_i\big) \leq 0 \Big\} \quad \forall~i \in I.
			\end{align*}
			This satisfies the second part of Definition~\ref{definition:BNE:compensation}. 
			
			By Assumption~\refAss{upper:level:convex:quadratic}, we know that $\partial \, G\big(\cdot\big) / \partial \, \zeta_i >0 \ \forall~i\in I$ and, hence, for any optimal solution, for each player,~$\big(\zeta_i\big)_{i \in I}$ is minimal. This satisfies the third part of Definition \ref{definition:BNE:compensation} and hence we have shown that if $\big(x_i,y_i,\widetilde{y}\upfixedx_i,\widetilde{\lambda}\upfixedx_i,\kappa\upfixedx_i,\zeta\upfixedx_i\big)_{i \in I}$ is an optimal solution to problem~\eqref{model:overall:problem}, then $\big(x_i,y_i\big)_{i \in I}$ with compensation payments $\big(\zeta_i\big)_{i \in I}$ is a solution to the binary game defined by Definition~\ref{definition:BNE:compensation}. 
			
			Now, we assume that $\big(x_i,y_i\big)_{i \in I}$ with compensation payments $\big(\zeta_i\big)_{i \in I}$  is a quasi-equilibrium to the binary game with compensation defined by Definition~\ref{definition:BNE:compensation}. 
			Choose~$\widetilde{K}$ large enough so that it is greater than the difference between any upper and lower bounds on~$y_i \ \forall~i \in I$ and greater than the difference between the minimum and maximum value of~$f_i \ \forall~i \in I$. Such a value exists since by Assumption~\refAss{KKT:conditions}, the feasible set is compact and $f_i$ is continuous, so the maximum and minimum must exist. 
			Then, by Definition~\ref{definition:BNE:compensation}, for any fixed value of $x_i$ and $y_{-i}$, $y_i$ is an optimal solution to the individual player's optimization problem. Thus, you can find $(\widetilde{y}\upfixedx_i,\widetilde{\lambda}\upfixedx_i)$ such that $\big(x_i,y_i,\widetilde{y}\upfixedx_i,\widetilde{\lambda}\upfixedx_i\big)_{i \in I}$ satisfies the first two constraints in problem~\eqref{model:overall:problem}. Calculate $\zeta\upfixedx_i$ from $\zeta_i \ \forall~i\in I$ and choose $\kappa\upfixedx_i$ according to the third constraint in problem~\eqref{model:overall:problem}. Thus, for any solution to the binary game in Definition~\ref{definition:BNE:compensation} given by  $\big(x_i,y_i\big)_{i \in I}$ and compensation payment $\big(\zeta_i\big)_{i \in I}$, we have shown that there there exists a vector~$\big(\widetilde{y}\upfixedx_i,\widetilde{\lambda}\upfixedx_i,\kappa\upfixedx_i,\zeta\upfixedx_i\big)_{i \in I}$, such that $\big(x_i,y_i,\widetilde{y}\upfixedx_i,\widetilde{\lambda}\upfixedx_i,\kappa\upfixedx_i,\zeta\upfixedx_i\big)_{i \in I}$ is a feasible point to problem~\eqref{model:overall:problem}. 
		\end{footnotesize}
	\end{proof}
\end{spacing}

\begin{spacing}{1.1}
	\begin{corollary} \label{corollary:feasible:minimal}
		Under Assumptions~\refAss{KKT:conditions} and~\refAss{upper:level:convex:quadratic}, any vector~$(x_i,y_i)_{i \in I}$ is a binary quasi-equilibrium for the Nash game in binary variables with compensation $\big(\zeta_i\big)_{i \in I}$ as defined in Definition~\ref{definition:BNE:compensation} if there exists a vector~$\big(\widetilde{y}\upfixedx_i,\widetilde{\lambda}\upfixedx_i,\kappa\upfixedx_i,\zeta\upfixedx_i\big)_{i \in I}$, such that $\big(x_i,y_i,\widetilde{y}\upfixedx_i,\widetilde{\lambda}\upfixedx_i,\kappa\upfixedx_i,\zeta\upfixedx_i\big)_{i \in I}$ is a feasible solution to problem~\eqref{model:overall:problem} and~$\zeta\upfixedx_i$ is minimal as defined in Definition~\ref{definition:BNE:compensation}. 
	\end{corollary}
	\begin{proof}
		\begin{footnotesize}
			By the arguments in Theorem~\ref{theorem:BNE:compensation}, for any point~$\big(x_i,y_i,\widetilde{y}\upfixedx_i,\widetilde{\lambda}\upfixedx_i,\kappa\upfixedx_i,\zeta\upfixedx_i\big)_{i \in I}$ that is feasible to problem \eqref{model:overall:problem}, the vectors~$\big(x_i,y_i)_{i \in I}$ and~$\big(\zeta_i)_{i \in I}$ satisfy the first two conditions of Definition~\ref{definition:BNE:compensation}. If, in addition, Condition~3 of Definition~\ref{definition:BNE:compensation} is satisfied, i.e., $\zeta_i$ is minimal for each player~$i \in I$, then~$\big(x_i,y_i\big)_{i \in I}$ is a binary quasi-equilibrium with compensation~$\big(\zeta_i\big)_{i \in I}$.
		\end{footnotesize}
	\end{proof}
\end{spacing}

If a vector is a global minimum to the objective function \eqref{model:overall:problem:objective}, this is the binary quasi-equilibrium with the lowest objective function value ~$F(\cdot) + G(\cdot)$. The following lemma and theorem provide conditions for the existence of a binary quasi-equilibrium that can be supported by appropriate compensation payments.

\begin{spacing}{1.1}
	\begin{lemma}[Existence of a Nash equilibrium in a game with fixed binary variables] \label{lemma:existence:fixed:binary}
		If for a fixed vector~$\big(\overline{x}_i\big)_{i \in I}$, the objective function~$f_i\big(\overline{x}_i,y_i,y_{-i}(x_{-i})\big)$ of every player~\mbox{$i \in I$} is continuous with regard to~$y_i$ and~$y_{-i}(\overline{x}_{-i})$, and quasi-convex in~$y_i$, and  the feasible region defined by the constraints~$g_i\big(\overline{x}_i,y_i\big)$ is compact, convex and non-empty,  then the resulting continuous game has a solution.
	\end{lemma}
	\begin{proof}
		\begin{footnotesize}
			The existence follows from Kakutani's fixed-point theorem \citep{Glicksberg_1952}.
		\end{footnotesize}
	\end{proof}
\end{spacing}

Relaxations of these conditions for the existence of a Nash equilibrium in continuous decision variables are also discussed in the literature \citep{Facchinei_2003_vi,Nishimura_Friedman_1981_Nash_quasiconcave}. The existence result for Nash equilibria in continuous games given fixed binary variables in Lemma~\ref{lemma:existence:fixed:binary} can be combined with Theorem~\ref{theorem:existence} to provide reasonable conditions for the existence of binary quasi-equilibria.
\begin{spacing}{1.1}
	\begin{theorem}[Existence of a binary quasi-equilibrium]\label{theorem:existence}
		Under Assumption~\refAss{KKT:conditions} and~\refAss{upper:level:convex:quadratic}, if for any fixed vector~$(\overline{x}_i)_{i \in I}$, the resulting continuous game has a solution, then a corresponding binary quasi-equilibrium exists for the Nash game in binary variables. 
	\end{theorem}
	\begin{proof}
		\begin{footnotesize}
			For any~$\big(y_i\big)_{i \in I}$ that is a Nash equilibrium given the fixed vector~$\big(\overline{x}_i\big)_{i \in I}$,
			we can find a vector~$\big(\widetilde{y}\upfixedx_i,\widetilde{\lambda}\upfixedx_i\big)_{i \in I}$ such that~$\big(\overline{x}_i,y_i,\widetilde{y}\upfixedx_i,\widetilde{\lambda}\upfixedx_i\big)_{i \in I}$ is a feasible point to constraints~(\ref{model:overall:problem:KKT:1}--\ref{model:overall:problem:KKT:constraints:0}). Recall that~$\fixedx=\big(\{\overline{x}_i,1-\overline{x}_i\}\big)_{i \in I}$. Choose~$\widetilde{K}$ as in the proof for Theorem~\ref{theorem:BNE:compensation}. 
			
			We can then find values the vector for~$\big(\kappa\upfixedx_i,\zeta\upfixedx_i\big)_{i \in I}$ such that $\big(\fixedx,\widetilde{y}\upfixedx_i,\widetilde{\lambda}\upfixedx_i, \kappa\upfixedx_i,$ $\zeta\upfixedx_i\big)_{i \in I}$ satisfy constraints~(\ref{model:overall:problem:incentive}--\ref{model:overall:problem:translate1}), and where either $\kappa\upfixedx_i=0$ or $\zeta\upfixedx_i=0$ for every player~\mbox{$i \in I$}. By equation~\eqref{model:overall:problem:incentive}, this implies that $\zeta\upfixedx_i$ is minimal. By Corollary~\ref{corollary:feasible:minimal}, $\big(\overline{x}_i,y_i\big)_{i \in I}$ with compensation payments $\big(\zeta_i\big)_{i \in I}$ is a binary quasi-equilibrium.
		\end{footnotesize}
	\end{proof}
\end{spacing}

Theorem~\ref{theorem:existence} implies that, if a Nash equilibrium solution to the continuous game exists for any fixed realization of the binary variables, then this solution can be supported as a quasi-equilibrium with appropriate compensation payments.

The reformulated multi-objective problem subject to a binary quasi-equilibrium \eqref{model:overall:problem} is a mixed-integer program. However, the incentive-compatibility constraint (condition~\ref{model:overall:problem:incentive}) can still cause numerical difficulties, because the players' objective functions are often not linear and not even convex in terms of all variables, even when they are linear from the point of view of the player itself. We will now introduce a special case, which allows to reduce problem~\eqref{model:overall:problem} to a linear or quadratic convex mixed-integer program with linear constraints.

\begin{itemize}
	\itemAss{lower:level:linear} Assume that each lower-level player's objective function~$f_i\big(x_i,y_i,y_{-i}(x_{-i})\big)$ can be separated into two functions, where the first part is linear with respect to~$x_i$ and~$y_{-i}$, and the second part is linear only with respect to~$y_i$,
	\begin{align*}
	f_i\big(x_i,y_i,y_{-i}(x_{-i})\big) = 
	f^{(x)}_i\big(x_i,y_{-i}(x_{-i})\big) +
	f^{(y)}_i\big(y_i\;|\;y_{-i}(x_{-i})\big),
	\end{align*}
	and the partial derivative of the objective function with regard to the continuous variable~$y_i$, $\nabla_{y_i} f^{(y)}_i\big(y_i\upfixedx\;|\;y_{-i}(x_{-i})\big)$, is linear in all variables.
	
	Furthermore, assume that all constraints $g_i\big(x_i,y_i\big)$ are affine and can therefore be written as:
	\begin{align*}
	g_i\big(x_i,y_i\big) \leq 0 \Rightarrow  a_i x_i + A_i y_i \leq b_i
	\end{align*}
	where $a_i,b_i$ are vectors and $A_i$ a matrix of suitable dimensions.
\end{itemize}
This assumption implies that the functions~$f^{(y)}_i\big(y_i \;|\; y_{-i}(x_{-i})\big)$ need not be linear with respect to all variables, only with regard to the player's own continuous decision variable~$y_i$. 

One consequence of Assumption \refAss{lower:level:linear} is the caveat that the following theorems are not directly applicable to \emph{Nash-Cournot equilibrium} models, where a player is aware of its own impact on the final demand price. These models are usually formulated such that a player faces an objective function of the form $\max_y p(y) \, y$, where $p(y)$ is the inverse demand curve; this violates the linearity requirement for the reformulation.

However, we only need Assumption \mbox{\refAss{lower:level:linear}} to prove Theorem 4 below, which allows us to write the problem as a mixed-integer quadratic program, and obtain global optimality results. Whenever Assumption A3 does not hold, as in general game-theoretic settings, the reformulation introduced above as well as Theorems~\mbox{\ref{theorem:BNE}} and~\mbox{\ref{theorem:BNE:compensation}} are still applicable, but we cannot mathematically prove global optimality of a numerical solution. If we can show that problem (14) can be solved to optimality without Assumption A3, our method will obtain globally optimal results to the binary equilibrium problem. However, this will require novel research into general game-theoretic settings as well as nonlinear, mixed-integer optimization problems, both of which we plan to address in future work. We proceed with the theoretical results below to justify the general setting of the power market uplift problem, which is the topic of this paper.

\begin{spacing}{1.1}
	\begin{theorem}[Exact reformulation as a mixed-integer linear/quadratic program with linear constraints]\label{theorem:BNE:linear}
		Under Assumptions~\refAss{KKT:conditions}, \refAss{upper:level:convex:quadratic} and~\refAss{lower:level:linear}, the multi-objective program subject to a binary quasi-equilibrium (problem~\ref{model:overall:problem}) can be reformulated as a quadratic integer program with linear constraints. Theorems~\eqref{theorem:BNE} and~\eqref{theorem:BNE:compensation} remain valid.
	\end{theorem}
	\begin{proof} 
		\begin{footnotesize}
			By Assumptions~\refAss{KKT:conditions}, \refAss{upper:level:convex:quadratic} and~\refAss{lower:level:linear}, the objective function is linear or convex quadratic, and the first-order conditions and the players' constraints are linear. The complementarity conditions~\eqref{model:overall:problem:KKT:constraints:1} and~\eqref{model:overall:problem:KKT:constraints:0} can be reformulated using disjunctive constraints \citep{Fortuny-Amat_1981}.
			
			In the next steps, we will show how the incentive-compatibility constraint~\eqref{model:overall:problem:incentive} can be reformulated as a linear constraint. First, by assumption \refAss{lower:level:linear}, the function~$f^{(y)}_i\big(y_i\upfixedx\;|\;y_{-i}(x_{-i})\big)$ is linear with respect to~$y_i\upfixedx$. We know that all linear functions can be written as the product of their partial derivative and the variable with which the partial derivative was taken. More specifically, this can be written as:
			\begin{align*}
			f^{(y)}_i\big(y_i\upfixedx\;|\;y_{-i}(x_{-i})\big) = \Big(\nabla_{y_i} f^{(y)}_i\big(y_i\upfixedx\;|\;y_{-i}(x_{-i})\big) \Big)^T y_i\upfixedx.
			\end{align*}
			
			By first-order optimality, $\Big(\nabla_{y_i} f^{(y)}_i\big(y_i\upfixedx\;|\;y_{-i}(x_{-i})\big) \Big)= - \Big(\big(\lambda_i\upfixedx\big)^T \nabla_{y_i} \, g_i\big(\fixedx,y_i\upfixedx\big) \Big)$.
			
			Then, applying the definition of~$g_i(\cdot)$ in Assumption~\refAss{lower:level:linear}, $\nabla_{y_i} \, g_i\big(\fixedx,y_i\upfixedx\big) = A_i$, and using the complementarity condition of constraints~\eqref{model:overall:problem:KKT:constraints:1} and~\eqref{model:overall:problem:KKT:constraints:0}, $(a_i x_i + A_i y_i - b_i)^T \lambda_i = 0$,
			the reformulation proceeds as follows:
			\begin{align*}
			f_i\big(\fixedx,y_i,y_{-i}(x_{-i})\big) &= f^{(x)}_i\big(\fixedx,y_{-i}(x_{-i})\big) 
			+ f^{(y)}_i\big(y_i\upfixedx\;|\;y_{-i}(x_{-i})\big) \\
			&= f^{(x)}_i\big(\fixedx,y_{-i}(x_{-i})\big) 
			+ \Big(\nabla_{y_i} f^{(y)}_i\big(y_i\upfixedx\;|\;y_{-i}(x_{-i})\big) \Big)^T y_i\upfixedx \\
			&= f^{(x)}_i\big(\fixedx,y_{-i}(x_{-i})\big) 
			- \Big(\big(\lambda_i\upfixedx\big)^T \nabla_{y_i} \, g_i\big(\fixedx,y_i\upfixedx\big)  \Big)^T y_i\upfixedx \\
			&= f^{(x)}_i\big(\fixedx,y_{-i}(x_{-i})\big) 
			- \Big(\big(\lambda_i\upfixedx\big)^T A_i \Big)^T y_i\upfixedx \\
			&= f^{(x)}_i\big(\fixedx,y_{-i}(x_{-i})\big) 
			- \big(\lambda_i\upfixedx\big)^T \big( b_i - a_i \fixedx \big)
			\end{align*}
			Therefore, problem \eqref{model:overall:problem} can be reformulated as a linear/quadratic convex integer program with linear constraints.
		\end{footnotesize}
	\end{proof}
\end{spacing}
With this theorem, we show that our approach can be applied to a large number of problem classes and still be solved as a mixed-integer linear program. These include operational constraints such as capacity bounds or minimum generation requirements, and market forms including linear inverse demand functions. 

\subsection{Comparing the binary equilibrium to a social-planner model}\label{sec:model:comparison}
As a benchmark for comparing the binary-equilibrium method to commonly used approaches, we apply the following method: first, solve the welfare-maximizing problem assuming a central planner, then derive prices from the linearized model using the \cite{ONeill_2005_nonconvexities} method, and finally compute compensation payments for each player that has a profitable deviation, to ensure that every market participant is at last indifferent and the outcome is stable in the Nash sense; i.e., no profitable deviation exists given rivals' actions and market prices. This method is, in a way, a lexicographic solution approach to the overall problem; mathematically, it can be seen as a hierarchical min-min problem. For illustration, we formulate this problem along the notation used in the derivation of the binary equilibrium method: 
\begin{align}
\min_{\zeta_i} \ \ G&\Big(\big(\zeta_i\big)_{i \in I}\Big)
~ + \left\{ 
\begin{array}{l}
\displaystyle\min_{x_i,y_i} \ F\Big(\big(x_i,y_i\big)_{i \in I}\Big) \\[6 pt]
~\st g_i\big(x_i,y_i\big) \leq 0 \quad  \forall~i \in I
\end{array} \right\}
\label{twostage:problem}
\\ 
\st &
\left\{ 
\begin{array}{l}
f_i\big(x_i, y_i,y_{-i}(x_{-i})\big) - \zeta_i \leq  f_i\big(x_i^\times, y_i^\times,y_{-i}(x_{-i})\big)   \\[4pt]
\quad \forall \ (x_i^\times, y_i^\times) \in \displaystyle\argmin_{(x_i^\diamond, y_i^\diamond)} 
\Big\{f_i\big(x_i^\diamond, y_i^\diamond,y_{-i}(x_{-i})\big) \;\big|\;g_i\big(x_i^\diamond,y_i^\diamond\big)  \leq 0  \Big\} 
\end{array} \right\} \nonumber
\end{align}
\hfill $\forall \ i \in I$

\noindent Throughout the following discussion, we will denote the optimal solution of the inner welfare optimization of problem~\eqref{twostage:problem} as~$\big(x_i^\circ,y_i^\circ\big)_{i \in I}$, while~$\big(\zeta_i^\circ\big)_{i \in I}$ is the minimal compensation payment to guarantee incentive compatibility for all players. Similar as in the definition of the binary game introduced earlier, the objective function of player~$i$ evaluated at $\big(x_i^\times,y_i^\times\big)$ is the best a player can do by deviating (cf. Definitions~\ref{definition:BNE} and~\ref{definition:BNE:compensation}). In the context of this example, deviation is to be understood as not following the decision of the benevolent social planner.

For comparison, we restate the multi-objective program subject to a binary quasi-equilibrium (problem~\ref{model:overall:problem}) in simplified form:
\begin{align}
\min_{\zeta_i} \ G&\Big(\big(\zeta_i\big)_{i \in I}\Big) + F\Big(\big(x_i,y_i\big)_{i \in I}\Big)
\tag{\ref*{model:overall:problem}'} \label{MOPBQE:problem}
\\ 
\st &
\left\{ 
\begin{array}{rcl}
\nabla_{y_i} \, f_i\Big(\fixedx,\widetilde{y}_i\upfixedx,y_{-i}\Big) + \big(\widetilde{\lambda}_i\upfixedx\big)^T \nabla_{y_i} \, g_i\Big(\fixedx,\widetilde{y}_i\upfixedx\Big) &=& 0 \\
0 \leq - g_i\Big(\fixedx,\widetilde{y}_i\upfixedx \Big) & \bot & \widetilde{\lambda}_i\upfixedx \geq 0  \\
f_i\Big(\one,y_i\upone,y_{-i}\Big) + \kappa_i\upone - \zeta_i\upone - \kappa_i\upzero + \zeta_i\upzero&=& f_i\Big(\zero,y_i\upzero,y_{-i}\Big) \\[4 pt]
\text{``translation'' constraints \ref{model:overall:problem:duals1}--\ref{model:overall:problem:translate1}}& \\
\end{array} \right\} \nonumber
\end{align}
\hfill $\forall \ i \in I$

\noindent In the following discussion, the optimal values of each player's binary and continuous decision variables in equilibrium and the compensation payments are denoted by~$\big(x_i^*,y_i^*,\zeta_i^*\big)_{i \in I}$. The key difference between the above mathematical programs is that problem~\eqref{twostage:problem} solves an optimization while problem~\eqref{model:overall:problem} solves for a non-cooperative equilibrium among all players. Thus, problem~\eqref{twostage:problem} is constrained by general power market constraints while problem~\eqref{model:overall:problem} has additional equilibrium constraints. Clearly, problem~\eqref{model:overall:problem} is a restricted version of problem~\eqref{twostage:problem} whenever there are no compensation payments or when the objective function does not include a penalty term~\smash{$G\big((\zeta_i)_{i \in I}\big)$}. 

The following two theorems formalize the idea that the social planner problem is a less constrained version of the binary equilibrium problem. First, we show that if the socially optimal outcome~$\big(x_i^\circ,y_i^\circ\big)_{i \in I}$ does not require any compensation payments, this solution (weakly) dominates the binary equilibrium outcome. Second, a similar argument can be made if profitable deviations exist in the socially optimal outcome, but the social planner does not assign a penalty for compensation in its objective function, i.e., $G(\cdot)=0$. This is the case if such payments are assumed to incur no loss to overall welfare. 

\begin{spacing}{1.1}
	\begin{theorem} \label{theorem:no_compensation}
		If any optimal solution obtained by the social planner (problem~\ref{twostage:problem}) does not require compensation payments to any player~$i \in I$, then the objective value achieved by the social planner is at least as small as solution of the binary quasi-equilibrium model, i.e.,
		\begin{align*}
		\sum_i \zeta_i^\circ = 0 
		\quad \Rightarrow & \quad
		f_i\big(x_i^\circ,y_i^\circ,y_{-i}^\circ(x_{-i}^\circ)\big) \leq  f_i\big(x_i^\times, y_i^\times,y_{-i}^\circ(x_{-i}^\circ)\big) \quad \forall \ i \in I \\
		\quad \Rightarrow & \quad
		F\Big(\big(x_i^\circ,y_i^\circ\big)_{i \in I}\Big) \leq F\Big(\big(x_i^*,y_i^*\big)_{i \in I}\Big).
		\end{align*}
	\end{theorem}
	\begin{proof}
		\begin{footnotesize}
			Assume $\sum_i \zeta_i^\circ = 0$ at any optimal point for problem~\eqref{twostage:problem}. Whenever $\sum_i \zeta_i^\circ = 0$ in problem~\eqref{twostage:problem}, the feasible region of problem~\eqref{model:overall:problem} is a subset of the feasible region for problem~\eqref{twostage:problem}. Moreover, since $G\big((\zeta_i)_{i \in I}\big)$ is an increasing function of $(\zeta_i)_{i \in I}$, whenever $\sum_i \zeta_i^\circ = 0$, $G\big((\zeta_i)_{i \in I}\big)$ is at its minimum. Thus, the optimal objective function value of problem~\eqref{twostage:problem} forms a lower-bound to any optimal objective function value of problem~\eqref{model:overall:problem}. These two facts combine to show that whenever  $\sum_i \zeta_i^\circ = 0$ at any optimal point for problem~\eqref{twostage:problem}, $F\big((x_i^\circ,y_i^\circ)_{i \in I}\big) \leq F\big((x_i^*,y_i^*)_{i \in I}\big)$.     
		\end{footnotesize}
	\end{proof}
\end{spacing}
The reasoning for Theorem~\ref{theorem:no_compensation} is quite straightforward: if compensation payments are not required, the regularizer $G\big((\zeta_i)_{i \in I}\big)$ does not add anything to the objective value of problem~\eqref{twostage:problem} and the incentive-compatibility constraints are not relevant. Then, the binary-equilibrium model (problem~\ref{model:overall:problem}) is a more constrained version of the inner problem of the social-welfare maximizing planner: the objective function~$F\big((x_i,y_i)_{i \in I}\big)$ and the constraints~$g_i\big(x_i,y_i\big)$ are present in both programs, but the binary-equilibrium reformulation adds further constraints.

\begin{spacing}{1.1}
	\begin{theorem} \label{theorem:free_compensation}
		If compensation payments do not incur any cost to overall societal welfare in problems~\eqref{twostage:problem}~and~\eqref{model:overall:problem}, then the optimal objective value achieved by the social planner is at least as small as the solution of the binary quasi-equilibrium model, i.e., 
		\begin{align*}
		G\Big(\big(\zeta_i\big)_{i \in I}\Big)=0 
		\quad \Rightarrow \quad
		F\Big(\big(x_i^\circ,y_i^\circ\big)_{i \in I}\Big) \leq F\Big(\big(x_i^*,y_i^*\big)_{i \in I}\Big).
		\end{align*}
	\end{theorem}
	\begin{proof}
		\begin{footnotesize}
			Assume that $G\big((\zeta_i)_{i \in I}\big)=0$ in both problems~\eqref{twostage:problem}~and~\eqref{model:overall:problem}. Then, for any $\big(\zeta_i\big)_{i \in I}$, the feasible region of problem~\eqref{model:overall:problem} is a subset of the feasible region of problem~\eqref{twostage:problem}. Thus, for any optimal solution to problem~\eqref{twostage:problem}, $F\big( (x_i^\circ,y_i^\circ)_{i \in I}\big) \leq F\big((x_i^*,y_i^*)_{i \in I}\big)$.  
		\end{footnotesize}
	\end{proof}
\end{spacing}

The reasoning is similar to Theorem~\ref{theorem:no_compensation}: if compensation does not incur any penalty or negative effects, then redistribution of the biggest possible surplus (or least possible cost) is indeed the preferred strategy. The additional constraints of the binary-equilibrium method may restrict the solution space and could make the socially optimal outcome infeasible in problem~\eqref{model:overall:problem}.

So far, we have shown cases where the social planner is superior to the binary-equilibrium model. We now show that if compensation payments are to be used, there is a case where the opposite is true.
\begin{itemize}
	\itemAss{two-stage:comparison} Assume that the objective function of the overall problem is the aggregate of the pay-offs of each player, i.e., \\
	$F\Big(\big(x_i,y_i\big)_{i \in I}\Big)= \sum_i f_i\big(x_i,y_i,y_{-i}(x_{-i})\big).$
\end{itemize}
This assumption means that the upper-level player acts as a social planner in the sense that it optimizes the joint pay-off of all market participants. It also implies that no player exerts market power (i.e., each firm is a price-taker). Another implication is that consumer welfare can be included in the objective function, if demand is modeled as an active player. 

\begin{spacing}{1.1}
	\begin{theorem}\label{theorem:eq_compensation}
		Under Assumptions~\refAss{KKT:conditions}, \refAss{upper:level:convex:quadratic}, \refAss{lower:level:linear} and~\refAss{two-stage:comparison}, the optimal solution of the multi-objective problem subject to a binary quasi-equilibrium (problem~\ref{model:overall:problem}) is at least as good as the solution to the social-welfare problem~\eqref{twostage:problem} subject to ex-post compensation to guarantee incentive compatibility, i.e.,
		\begin{align*}
		F\Big(\big(x_i^*,y_i^*\big)_{i \in I}\Big) + G\Big(\big(\zeta_i^*\big)_{i \in I}\Big)  \leq F\Big(\big(x_i^\circ,y_i^\circ\big)_{i \in I}\Big) + G\Big(\big(\zeta_i^\circ\big)_{i \in I}\Big).
		\end{align*}
	\end{theorem}
	\begin{proof} 
		\begin{footnotesize}
			Let $(x_i^\circ,y_i^\circ)_{i \in I}$ denote the optimal solution to the social-welfare problem, and $(\zeta_i^\circ)_{i \in I}$ is the vector of compensation payments necessary to align the incentives of all players. Following Assumption~\refAss{two-stage:comparison}, we know that the first-order optimality conditions of the social planner in problem~\eqref{twostage:problem} are identical to the stacked first-order optimality conditions of each player, whenever the binary variables are fixed at $x_i^\circ$, i.e.,
			\begin{align*}
			\nabla_{y_i} \, F_i\big(x_i^\circ,y_i\big) 
			+ \big(\lambda_i\big)^T \nabla_{y_i} \, g_i\big(x_i^\circ,y_i\big) 
			= \left\{
			\nabla_{y_i} \, f_i\big(x_i^\circ,y_i\big) + \big(\lambda_i\big)^T \nabla_{y_i} \, g_i\big(x_i^\circ,y_i\big)
			\right\}_{i \in I}
			\end{align*}
			Therefore, the vector $(y_i^\circ)_{i \in I}$ is an equilibrium to the continuous game with binary variables fixed at~$(x_i^\circ)_{i \in I}$, and following Theorem~\ref{theorem:existence}, the vector~$(x_i^\circ,y_i^\circ)_{i \in I}$ can be implemented as a binary quasi-equilibrium. The outcome~$(x_i^\circ,y_i^\circ)_{i \in I}$ with compensation vector~$(\zeta_i^\circ)_{i \in I}$ and corresponding dual variables thus feasible to problem~\eqref{model:overall:problem}.
		\end{footnotesize}
	\end{proof}
\end{spacing}

The above theorem has an important implication for the numerical implementation of the binary-equilibrium method: if Assumption~\refAss{two-stage:comparison} is satisfied, the solution to the social-welfare optimization program (problem~\ref{twostage:problem}) can be used as a starting point for solving the multi-objective binary-equilibrium problem. As the social-welfare problem is less computationally challenging, this should help in developing more efficient solution methods for the binary equilibrium method.

Last, we state conditions when the solutions of the two approaches coincide.
\begin{spacing}{1.1}
	\begin{corollary}
		Under Assumptions~\refAss{KKT:conditions}, \refAss{upper:level:convex:quadratic}, \refAss{lower:level:linear} and~\refAss{two-stage:comparison}, whenever either $\sum_i \zeta_i^\circ = 0$ at optimality for problem~\eqref{twostage:problem} or $G\big((\zeta_i)_{i \in I}\big)=0$  for both problems~\eqref{twostage:problem}~and~\eqref{model:overall:problem} the optimal objective function value of the multi-objective problem subject to a binary quasi-equilibrium (problem~\ref{model:overall:problem}) is equal to the optimal objective function value of the social-welfare problem~\eqref{twostage:problem} subject to ex-post compensation to guarantee incentive compatibility, i.e.,
		\begin{align*}
		F\Big(\big(x_i^*,y_i^*\big)_{i \in I}\Big) + G\Big(\big(\zeta_i^*\big)_{i \in I}\Big)  
		= 
		F\Big(\big(x_i^\circ,y_i^\circ\big)_{i \in I}\Big) + G\Big(\big(\zeta_i^\circ\big)_{i \in I}\Big) .
		\end{align*}
		
	\end{corollary}
	\begin{proof} 
		\begin{footnotesize}
			The proof follows from combining the results of Theorems~\ref{theorem:no_compensation}~and~\ref{theorem:free_compensation} with the results from Theorem~\ref{theorem:eq_compensation}.
		\end{footnotesize}
	\end{proof}
\end{spacing}
In conjunction with Theorems~\ref{theorem:no_compensation}~and~\ref{theorem:free_compensation}, Theorem~\ref{theorem:eq_compensation} has an important implication: if no compensation is necessary at the welfare-optimal solution and the assumptions for the theorem are satisfied, the two solutions coincide. However, if compensation is necessary to mitigate deviation incentives, then the multi-objective program subject to a binary equilibrium can yield a better overall result. This is because the multi-objective program can incorporate the trade-off between welfare and the compensation payments necessary to guarantee a stable Nash equilibrium. 

\subsection{Computational complexity of the binary quasi-equilibrium}\label{sec:model:computational}
Last, let us compare the mathematical complexity of the social-welfare approach (problem~\ref{twostage:problem}) to our method in a linear problem setting, i.e.,~Assumption~\refAss{lower:level:linear} holds and the upper-level objective function is linear. Here, we focus on comparing the number of binary variables in each approach. The number of constraints and continuous variables also influence the computational complexity.

Alas, the exponential increase of complexity in the number of binary variables is a more serious concern than additional constraints or continuous variables in a mixed-integer program, hence our focus on this aspect in the following discussion. 

The two-stage approach following \cite{ONeill_2005_nonconvexities} requires first solving a linear mixed-binary optimization problem with $p\,n$~binary variables (where $p$~is the number of players and $n$~is the number of binary decision variables of each player), and then solving $p$~optimization problems with $n$~binary variables to determine the optimal alternative for each player. 

In contrast, the multi-objective program subject to a binary quasi-equilibrium (problem~\ref{model:overall:problem}) requires to solve a mixed-binary optimization problem with $p \, (1+k) \,2^n$~binary variables, where $k$~is the number of inequality constraints per player. The approach requires $2^n$~binary variables to represent each permutation of binary decisions per player.\footnote{For an illustration of a more general application of the binary equilibrium method, we refer to the natural gas market investment and production game available for download at \url{https://github.com/danielhuppmann/binary_equilibrium}.} The term~$k\, 2^n$~in the number of binary variables is due to the disjunctive constraints reformulation to replace the complementarity conditions, which every player has to consider for all permutations of her options in binary variables. This is necessary to compute the optimal value of the continuous variables for each permutation. 

This curse of dimensionality is similar to the numerical caveats of applying brute-force enumeration strategies to solve binary games, which requires to solve $2^{pn}$~continuous problems. However, we illustrate in the numerical application presented in the following section that the number of binary variables in our method can be significantly reduced depending on the underlying problem. 

In the power market uplift problem, the number of binary variables to obtain an exact solution is only~$(n+k) \, p + l$, where~$l$ is the number of binary variables required for the ISO; this is, in principle, not a substantial increase in computational complexity compared to the~$p \, n$ binary variables in the social-planner problem plus an additional $p$~optimization problems with $n$~binary variables each to determine the compensation payments required for incentive-compatibility to be satisfied for each generator. In short, our approach scales well in the number of players, but not necessarily in the number of binary variables of each player. We discuss further practical aspects on the computational aspects of the problem in Section \ref{sec:example:computation}.

\begin{table}
	\begin{center}
		\begin{small}
			\begin{tabular}{p{0.1 cm}p{1.9 cm}@{ ... }p{8.5 cm}}
				\hline \hline
				\multicolumn{3}{l}{\textbf{Sets \& Mappings}} \\
				&$n,m \in N$	& nodes \\
				&$t \in T$		& time step, hours \\
				&$i \in I$		& generators, power plant units \\
				&$j \in J$		& load, demand units \\
				&$l \in L$		& power lines \\
				&$i \in I_n, j \in J_n$ & generator/load unit mapping to node $n$ \\
				&$n(i), n(j)$ & node mapping to generator $i$/load unit $j$ \\
				&$\phi \in \Phi$	& set of dispatch options (schedules) for each generator \\
				&$t \in T_\phi$	& hours in which a generator is active in dispatch option $\phi$ \\
				\hline
				\multicolumn{3}{l}{\textbf{Primal variables}} \\
				&$x_{ti}$		& on/off decision for generator $i$ in hour $t$ \\
				&$z\on_{ti},z\off_{ti}$ & inter-temporal start-up/shut-down decision \\
				&$y_{ti}$		& actual generation by generator $i$ in hour $t$ \\
				&$y\on_{ti}$		& generation if binary variable is fixed at $\fixedx$ \\
				&$d_{tj}$		& demand by unit $j$ in hour $t$ \\
				&$\delta_{tn}$	& voltage angle \\
				\hline
				\multicolumn{3}{l}{\textbf{Dual variables}} \\
				&$\alpha\on_{ti},\beta\on_{ti}$	& dual to minimum activity/maximum generation capacity \\
				&$\nu_{tj}$			& dual to maximum load constraint \\
				&$\mu^+_{tl},\mu^-_{tl}$ & dual to voltage angle band constraints \\
				&$\xi^+_{tn},\xi^-_{tn}$ & dual to thermal line capacity constraints \\
				&$\gamma_t$	& dual to slack bus constraints \\
				\hline
				\multicolumn{3}{l}{\textbf{Prices, switch value, and compensation variables}} \\
				&$p_{tn}$		& locational marginal price \\
				&$\kappa\on_{ti},\kappa\off_{ti}$	& switch value (defined per time step)\\
				&$\zeta_{i}$ & compensation payment (defined over entire time horizon) \\
				\hline
				\multicolumn{3}{l}{\textbf{Parameters}} \\
				&$c^G_{i}$		& linear generation costs \\
				&$c\on_i,c\off_i$		& start-up/shut-down costs \\
				&$c^D_{\phi}$		& commitment costs in dispatch option $\phi$ (start-up, shut-down)\\
				&$g^{min}_i$		& minimum activity level if power plant is online \\
				&$g^{max}_i$		& maximum generation capacity \\
				&$x^{init}_i$		& power plant status at start of model horizon ($t=0$)\\
				&$u^D_{tj}$			& utility of demand unit $j$ for using electricity \\
				&$d^{max}_{tj}$	& maximum load of unit $j$ \\
				&$f^{max}_l$		& thermal capacity of power line $l$ \\
				&$B_{nk},H_{lk}$			& line/node susceptance/network transfer matrices \\
				\hline \hline
			\end{tabular}
		\end{small}
		\caption{Notation for the nodal power market problem} \label{table:example:notation}
	\end{center}
\end{table}

\section{A binary game: The power market uplift problem} \label{sec:example}
As an illustration of the methodology to solve binary games with compensation, we present a stylized model of a nodal-pricing electricity market comprised of generators with binary decision variables, demand units and a network. 
An upper-level player assigns locational marginal prices to maximize short-run market efficiency \citep[cf.][]{Hobbs_2001_Poolco}. Second, she disburses compensation payments to align the incentives of generators with the overall (societally most beneficial) outcome. 

In contrast to the standard implementation of an Independent System Operator (ISO), the market operator in our setting explicitly takes into account the incentives of generators to deviate from the welfare-optimal dispatch schedule and assigns compensation payments to ensure that the selected solution is incentive-compatible. This is in contrast to the more common market setup, where generators only receive a "no-loss" compensation (guarantee of non-negative profits).

\subsection{The power market model}
The lower level of the market is composed of generators and a player representing consumers (demand) and the network operator. Table~\ref{table:example:notation} provides a summary of the notation used in the example.

\subsubsection*{The generators}
Each generator~$i \in I$ seeks to maximize profits from generating and selling electricity over the time horizon $t \in T$. For consistency with the previous chapter, the generator's optimization problem is written in minimization form:
\begin{subequations} \label{example:generator:optimization}
	\begin{align} 
	\min_{x_{ti},y_{ti},z\on_{ti},z\off_{ti}} &\quad  - p_{tn(i)} y_{ti} + c^G_{i} y_{ti} + c\on_{i} z\on_{ti} + c\off_{i} z\off_{ti}  \label{example:generator:objective} \\
	&\st x_{ti} g_i^{min} \leq y_{ti} \leq x_{ti} g_i^{max} \quad \big(\alpha\on_{ti},\beta\on_{ti}\big)  \label{example:generator:con:generation} \\
	& \hspace{0.8 cm} x_{ti} - x_{(t-1)i} = z\on_{ti} - z\off_{ti} \label{example:generator:con:intertemporal} \\
	& \hspace{1 cm} x_{ti} \in \{0,1\}, \quad y_{ti},z\on_{ti},z\off_{ti}  \in \mathbb{R}_+ \nonumber
	\end{align}
\end{subequations}

The linear generation costs are given by~$c^G_i$, the (binary) start-up costs are given by~$c\on_i$, and the (binary) shut-down costs are given by~$c\off_i$. The operation schedule is denoted by $x_{ti}$, the decision how much electricity to generate and sell to the grid is~$y_{ti}$. The ramping decisions in a particular period are denoted by~$z\on_{ti}$ (start-up) and~$z\off_{ti}$ (shut-down), respectively. The status of a power plant at the beginning of the model horizon~(i.e.,~$x_{0i}$) is given by the parameter~$x^{init}_i$.

The first set of constraints~\eqref{example:generator:con:generation} concerns the maximum generation capacity and the minimum activity level~$(g_i^{min},g_i^{max})$, if the power plant is operating~\mbox{$(x_{ti}=1)$}.
The shadow variables~$(\alpha\on_{ti},\beta\on_{ti})$ are only meaningful given a fixed operation schedule~$x_{ti}$, and we only compute them if the power plant is operational. If the power plant is not switched on, generation is equal to zero; however, due to costs incurred by shutting down the plant (assuming it was operational in the previous period or at the beginning of the model horizon), total profits may be negative even when the plant is not generating electricity. The second constraint~\eqref{example:generator:con:intertemporal} concerns the inter-temporal consideration, i.e., the decision in which time periods the power plant is operational: while the start-up and shut-down variables~$z\on_{ti}$ and~$z\off_{ti}$ are binary in nature, they can be relaxed to positive real numbers without loss of information. Integrality of these variables is automatically enforced by the on/off variables~$x_{ti}$. 

The market clearing price~$p_{tn}$ is the vector of locational marginal prices over time, where~$n(i)$ denotes the node at which generator~$i$ is located; the set~$n,m \in N$ denotes all nodes in the network. The prices arise as the duals of the market clearing condition introduced below, and each generator takes the price at her node as given.
Assuming that the power plant~$i$ is operating in period~$t$, the optimal amount of power generated~$y\on_{ti}$ and the dual variables associated with the constraints can be determined by solving the generator's first-order optimality (KKT) conditions:
\begin{subequations} \label{example:generator:KKT}
	\begin{align} 
	0 = c^G_{i} - p_{tn(i)} + \beta\on_{ti} - \alpha\on_{ti} \quad &\botcomma \quad y\on_{ti} \free \\
	0 \leq - g^{min}_{ti} + y\on_{ti} \quad & \bot \quad \alpha\on_{ti} \geq 0  \label{example:generator:KKT:min}\\
	0 \leq g^{max}_{ti} - y\on_{ti} \quad & \bot \quad \beta\on_{ti} \geq 0 \label{example:generator:KKT:max}
	\end{align}
\end{subequations}

Otherwise, the amount generated is zero, and we do not require the dual variables in that case. Hence, in contrast to the general formulation in Section~\ref{sec:model}, we can omit the KKT conditions in this example for the case that the power plant is not operating in period~$t$.

\subsubsection*{Demand for electricity and network constraints}
The other side of the market is a player seeking to maximize the welfare (utility) of consumers while guaranteeing feasibility of the transmission system, given locational marginal prices~$p_{tn}$.
A set of units~$j \in J$ consume electricity (load~$d_{tj}$), each located at a specific node~$n(j)$. The sets~$I_n$ and~$J_n$ are the generators and load units located at node~$n$, respectively. There are a set of power lines~$l \in L$ connecting the nodes; the direct-current load flow (DCLF) characteristics are captured using the susceptance matrix~$B_{nm}$ (node-to-node) and network transfer matrix~$H_{nl}$ (node-to-line mapping). This approach is equivalent to a power transfer distribution factor (PTDF) matrix. 

This player maximizes the utility of demand from using electricity~$d_{tj}$, where the per-unit utility is given by~$u^D_{tj}$. The first constraint is the upper bound on demand, as a load unit cannot use more than~$d^{max}_{tj}$~units of electricity.

The next constraints~(\ref{example:ISO:optimization:flow:pos},\ref{example:ISO:optimization:flow:neg}) ensure that network flows are feasible and the thermal capacity~$f^{max}_l$ of each power line is observed. Constraints~(\ref{example:ISO:optimization:angle:pos},\ref{example:ISO:optimization:angle:neg}) guarantee that the voltage angle~$\delta_{tn}$ is within the range~$[-\pi,\pi]$. The $B$-$H$-formulation requires to define one arbitrary node as \emph{slack bus}~$\hat{n}$, at which the voltage angle~$\delta_{t\hat{n}}$ is zero by assumption (constraint~\ref{example:ISO:slackbus}). In line with the previous notation, we write the objective function as a minimization problem. 
The term $\sum_{m} B_{nm} \delta_{tm}$ is the net injection of electricity into the grid at node $n$, which depends on the voltage angles $\delta_{tm}$; multiplying this term by the nodal price $p_{tn}$ and summing over all nodes yields the aggregate congestion rents of the system. The market clearing condition (demand, generation, and net injection) will be formally introduced later as a constraint of the upper-level player.
\begin{subequations}\label{example:ISO:optimization}
	\begin{align}
	\min_{d_{tj}, \delta_{tn}} \quad \sum_{j \in J} p_{tn(j)} \big(d_{tj} + \sum_{m \in N} B_{nm} \delta_{tm} \big) - u^D_{tj} d_{tj} & \hspace{3.8 cm} \label{example:ISO:optimization:objective} \\
	\st d^{max}_{tj} - d_{tj} \geq 0 &\quad (\nu_{tj}) \label{example:ISO:optimization:demand:max} \\[4 pt]
	f^{max}_l - \sum_{n \in N} H_{ln} \delta_{tn} \geq 0 &\quad (\mu^+_{tl}) \label{example:ISO:optimization:flow:pos} \\[-2 pt]
	f^{max}_l + \sum_{n \in N} H_{ln} \delta_{tn} \geq 0 &\quad (\mu^-_{tl})\label{example:ISO:optimization:flow:neg} \\[-4 pt]
	\pi - \delta_{tn} \geq 0 &\quad (\xi^+_{tn}) \label{example:ISO:optimization:angle:pos} \\[2 pt]
	\pi + \delta_{tn} \geq 0 &\quad (\xi^-_{tn}) \label{example:ISO:optimization:angle:neg} \\[2 pt]
	\delta_{t\hat{n}} = 0  & \quad (\gamma_t) \label{example:ISO:slackbus}
	\end{align}
\end{subequations}

Because the decision variables of demand and voltage angle are continuous, this problem can be solved simultaneously with the generators' problems using first-order optimality conditions:
\begin{subequations}\label{example:ISO:KKT}
	\begin{align}
	0 \leq - u^D_{tj} + p_{tn(j)} + \nu_{tj} &\quad \bot \quad d_{tj} \geq 0 \label{example:ISO:KKT:demand} \\[2 pt]
	0 = \sum_{m\in N} B_{mn} p_{tm} 
	+ \sum_{l \in L} H_{ln} \big( \mu^+_{tl} - \mu^-_{tl}\big) \quad & \nonumber \\[- 4 pt]
	+ \xi^+_{tn} - \xi^-_{tn} - 
	\genfrac\{\}{0pt}{0}{\gamma_t \quad \text{if }n = \hat{n}}{0 \quad \text{else} \hspace{0.55 cm}}
	& \quad \botcomma \quad \delta_{tn} \free \label{example:ISO:KKT:delta} \\
	0 \leq d^{max}_{tj} - d_{tj} &\quad \bot \quad \nu_{tj} \geq 0 \label{example:ISO:KKT:demand:max} \\
	0 \leq f^{max}_l - \sum_{n \in N} H_{ln} \delta_{tn} &\quad \bot \quad \mu^+_{tl} \geq 0 \label{example:ISO:KKT:flow:pos} \\
	0 \leq f^{max}_l + \sum_{n \in N} H_{ln} \delta_{tn} &\quad \bot \quad  \mu^-_{tl} \geq 0 \label{example:ISO:KKT:flow:neg} \\
	0 \leq \pi - \delta_{tn} &\quad \bot \quad \xi^+_{tn} \geq 0 \label{example:ISO:KKT:angle:pos} \\
	0 \leq \pi + \delta_{tn} &\quad \bot \quad \xi^-_{tn} \geq 0 \label{example:ISO:KKT:angle:neg} \\[2 pt]
	0 = \delta_{t\hat{n}}  &\quad \botcomma \quad \gamma_t \free \label{example:ISO:KKT:slackbus}
	\end{align}
\end{subequations}

\subsubsection*{The upper-level market operator}
As introduced in Section~\ref{sec:model}, an upper-level player acts as equilibrium coordination mechanism: she sets a short-term locational marginal price~$p_{tn}$ resulting from the market-clearing condition (linking constraint across players) as well as compensation payments to align the incentives of market participants and ensure that no player has a profitable deviation. Mathematically, this player forms the upper level of a two-stage, hierarchical game; the lower level is the binary quasi-equilibrium between the generators and the demand/network player.

The upper-level player's objective function is closely related to the generators and the demand/network player, but not identical -- the market operator does not only consider short-term market efficiency (i.e., the sum of all players' objective functions), but includes the welfare loss from the disbursement of compensation payments~$\zeta_i$. The objective function satisfies Assumptions~\refAss{upper:level:convex:quadratic} and~\refAss{two-stage:comparison}, and the individual player's problems satisfy Assumptions~\refAss{KKT:conditions} and~\refAss{lower:level:linear}.
\begin{subequations}
	\begin{align}
	\min \ \sum_{t \in T} \bigg[&
	\sum_{i \in I} c^G_{i} y_{ti} + c\on_{i} z\on_{ti} + c\off_{i} z\off_{ti}
	- \sum_{j \in J} u^D_{tj} d_{tj} \label{example:MO:optimization}
	\bigg] 
	+ \sum_{i \in I} \zeta_i \\
	\st & \sum_{j \in J_n} d_{tj} - \sum_{i \in I_n} y_{ti} + \sum_{m \in N} B_{nm} \delta_{tm} = 0 \label{example:MO:MBC} \\ 
	&\text{KKT conditions of demand and network feasibility (equations \ref{example:ISO:KKT})} \nonumber \\
	&\text{KKT conditions of the generators (equations \ref{example:generator:KKT}}) \nonumber \\
	&\text{binary equilibrium between generators (equations \ref{example:BNE:generator}, specified below)} \nonumber 
	\end{align}
\end{subequations}

The first part of the objective function is the sum of the generators' incurred costs and the utility of load units from using electricity; this is equivalent to~$F(\cdot)$ in the theoretical formulation (problem~\ref{model:overall:problem}). The second part is the regularizer~$G(\cdot)$, although it has a distinct interpretation in this example: because the compensation payments to generators have to be funded through fees on market participants or from general taxation, they usually involve some efficiency loss from market distortions.

The market operator sets a spot price $p_{tn}$, which is considered by the generators and the demand/network player in their respective optimization problem, such that the nodal energy balance constraint \mbox{\ref{example:MO:MBC}} is satisfied.

Let us now turn to the equations necessary to guarantee the binary equilibrium between the generators. If~$x_{ti}=0$ (i.e., the power plant of generator~$i$ is switched off in period~$t$), the first-order conditions can be omitted altogether; both the generation level and the short-term profits in this case are zero, and the fixed costs from starting up or shutting down will be included in the incentive-compatibility constraint. This leaves the KKT conditions of the generators (equations~\ref{example:generator:KKT}) to determine the optimal, short-term dispatch in the case that the generator is operating in this period ($x_{ti}=1$). The inter-temporal constraint of the power plant operation status has to be considered (constraint~\ref{example:BNE:generator:intertemporal}); the start-up and shut-down variables~$\big(z_{ti}\on,z_{ti}\off\big)$ are determined by the unit commitment variables~$x_i$. 

Next, let us turn to the incentive-compatibility constraint: we now have multiple time periods and we formulate the incentive-compatibility constraint in a different way than in the general formulation in Section~\ref{sec:model} (equation~\ref{model:overall:problem:incentive}). This is due to the problem that it is not obvious how to allocate start-up costs over multiple time periods. It is therefore preferable to define the compensation payments $\zeta_i \in \mathbb{R_+}$ over the entire model horizon and also write the incentive-compatibility constraint in this way, rather than as a period-by-period constraint. As a consequence, we also change the interpretation of~$\big(\kappa\on_{ti},\kappa\off_{ti}\big) \in \mathbb{R}$: it now represents the short-term profits or losses (revenue less generation costs), and it is not any more restricted to positive values, in contrast to the switch value in the previous section and the overview in Table~\ref{table:incentive_alignment}. 

\begin{subequations} \label{example:BNE:generator}
	\begin{align}
	x_{(t-1)i} + z\on_{ti} - z\off_{ti} &= x_{ti} \label{example:BNE:generator:intertemporal} \\
	\beta\on_{ti} g^{max}_{ti} - \alpha\on_{ti} g^{min}_{ti} - \kappa\on_{ti} + \kappa\off_{ti} &=0 \label{example:BNE:generator:profits} \\
	|\kappa\on_{ti}| &\leq x_{ti} \, \widetilde{K} \label{example:BNE:kappa_on} \\
	|\kappa\off_{ti}| &\leq (1-x_{ti}) \, \widetilde{K} \label{example:BNE:kappa_off} \\
	\sum_{t \in T} \bigg[\kappa\on_{ti}
	- c\on_{i} z\on_{ti} -  c\off_{i} z\off_{ti}  \bigg] + \zeta_{i}
	&\geq \sum_{t \in T_\phi} \bigg[\beta\on_{ti} g^{max}_{ti} - \alpha\on_{ti} g^{min}_{ti} \bigg]
	- c^D_{\phi i} \quad \nonumber \\ & \hspace{4.5 cm} \forall~\phi \in \Phi \label{example:BNE:generator:incentive}
	\end{align}
	The last constraint (equation~\ref{example:BNE:generator:incentive}) is the  incentive-compatibility condition: it ensures that the profits for each generator (per period, short-term profits or losses from generating less the commitment costs) in the actual market outcome plus the compensation (left-hand side) are greater than the profits which that player could earn in any other dispatch schedule~$\phi$ (right-hand side). The revenues for each dispatch schedule can be computed from the duals to the maximum generation and minimum activity constraints, summing over the periods in which the generator is operational according to schedule~$\phi$; these periods are collected in the set~\mbox{$T_\phi \subseteq T$}. The total commitment costs for each dispatch schedule are denoted by~$c^D_{\phi i}$.
	
	The final set of constraints of the binary quasi-equilibrium ``translates'' the optimal generation decision for both states of the binary variable $(y\on_{ti},0)$ into the generation level which is actually realized in equilibrium, $y_{ti}$.
	\begin{align}
	0 &\leq y_{ti} \leq x_{ti} g^{max}_{ti} \\
	y\on_{ti} - (1-x_{ti}) \, g^{max}_{ti} &\leq y_{ti} \leq y\on_{ti} + (1-x_{ti}) \, g^{max}_{ti}
	\end{align}
\end{subequations}
As stated in the previous section, the binary variables have an additional role in this formulation relative to a standard unit-commitment model: rather than simply stating whether a plant is operating or not, they control which of the two potential states with regard to the continuous variables are active and realized in equilibrium~($y_{ti}=y\on_{ti}$ if~$x_{ti}=1$, $y_{ti}=0$ otherwise). Furthermore, it ensures that the variables capturing the short-term profit~$\big(\kappa\on_{ti},\kappa\off_{ti}\big)$ are correctly assigned. 

\subsection{Alternative rules for compensation payments}\label{sec:example:alternative}
The incentive-compatibility constraint as stated above (equation~\ref{example:BNE:generator:incentive}) is the direct extension of constraint~\eqref{model:incentive:compatibility:reformulated} in a multi-period setting. The short-term profits or losses are succinctly captured by the vector~$\big(\kappa\on_{ti},\kappa\off_{ti}\big) \in \mathbb{R}$, and the start-up and shut-down costs are linear terms. As a consequence, this method is flexible and allows to easily implement a wide range of market rules regarding compensation disbursements. To illustrate the versatility of the approach and its applicability to different market designs, we formulate two alternative versions of the model.

First, we implement a \emph{no-loss} rule to replace the incentive-compatibility constraints: no generator may earn negative profits (i.e., lose money out of pocket):
\begin{align*}
\sum_{t \in T} \big[\kappa\on_{ti}
- c\on_{i} z\on_{ti} - c\off_{i} z\off_{ti} \big] + \zeta_{i}
&\geq 0 \tag{\ref*{example:BNE:generator:incentive}'} \label{example:BNE:generator:incentive:noloss}
\end{align*}
In this setting, there is no constraint stating that the dispatch selected by the ISO has to be incentive-compatible for each generator. Instead, every market participant is forced to follow the schedule selected by the market operator. Hence, this setting omits the game-theoretic considerations. Furthermore, the power plant is compensated even if it would incur losses irrespective of the selected dispatch, so there may be over-compensation. This can happen in our setting because there are shut-down costs and some generators are operational at the beginning of the model horizon.

The second rule stipulates that only power plants can receive compensation payments if they were active at least once over the model horizon; this is to reflect the potential concern that no generator should receive compensation for doing nothing.
\begin{align*}
\zeta_{i}&\leq \sum_{t \in T} x_{ti}\, \widetilde{K} \tag{\ref*{example:BNE:generator:incentive}''} \label{example:BNE:generator:comp4act}
\end{align*}

\subsection{Illustrative results}
The power system adapted from \cite{Gabriel_2013_DCMLCP} consists of 6~nodes, with 9~generators and 4~load units (see Figure~\ref{figure:6_node_network}). Each generator has a maximum generation capacity of $100\,$MW and a minimum generation level, if operating, of $50\,$MW. Because we assume that one time period $t$ lasts for one hour, one unit of capacity~(MW) corresponds to one unit of energy~(MWh). Generators~$g3$ to~$g6$ are operational at the beginning of the model horizon and the power plants differ with regard to start-up, shut-down, and marginal generation costs. Demand for electricity varies over time, with a high utility for energy (or large willingness-to-pay, WTP) in the first hour and lower WTP in the second hour. All data for generators and load units are provided in Table~\ref{table:data}. Regarding the multi-objective function representing the market operator, we assume that each monetary unit paid in compensation is a one-for-one loss of welfare (sum of consumer utility, generator profits, and congestion rents). All lines have a thermal capacity of~$300\,$MW, except for the two inter-connector lines~$n2-n4$ and~$n3-n6$, which have a reduced thermal capacity of~$20\,$MW. Due to these bottlenecks, the standard, welfare-optimal unit commitment model yields losses for some generators.

\begin{figure}[b]
	\begin{center}
		\includegraphics[width= 0.7 \textwidth]{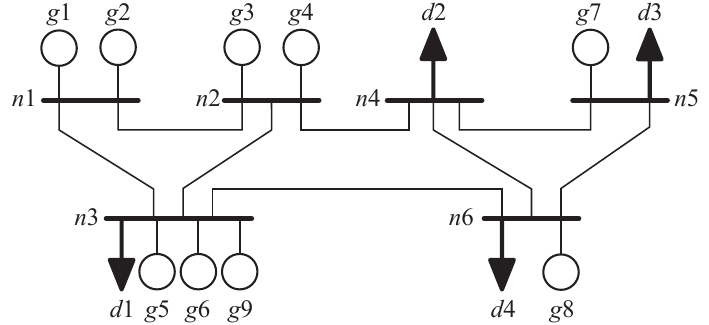}
	\end{center}
	\caption{6-node network, adapted from \cite{Gabriel_2013_DCMLCP},~p.~18}
	\label{figure:6_node_network}
\end{figure}

We compare three different \emph{market rule cases}; the game-theoretic considerations and regulations (constraints) concerning losses and compensations are repeated here for clarity:
\begin{itemize}
	\item[\textbf{Binary equilibrium:}] Every generator receives compensation such that she has no profitable deviation. The constraint is repeated here for easier comparison with the other market designs.
	\begin{align*}
	\sum_{t \in T} \bigg[\kappa\on_{ti}
	- c\on_{i} z\on_{ti} -  c\off_{i} z\off_{ti}  \bigg] + \zeta_{i}
	\geq \sum_{t \in T_\phi} \bigg[\beta\on_{ti} g^{max}_{ti} - &\alpha\on_{ti} g^{min}_{ti} \bigg]  - c^D_{\phi i} \nonumber \\
	& \forall~i\in I, \phi \in \Phi \tag{cf.~\ref*{example:BNE:generator:incentive}} 
	\end{align*}
	
	\item[\textbf{No-loss rule:}] We solve the multi-objective program (problem~\ref{example:MO:optimization}) subject to the constraint that no player earns negative profits (instead of constraint \ref{example:BNE:generator:incentive}). Players may have profitable deviations, for which they are not compensated.
	\begin{align*}
	\sum_{t \in T} \big[\kappa\on_{ti}
	- c\on_{i} z\on_{ti} - c\off_{i} z\off_{ti} \big] + \zeta_{i}
	&\geq 0  \quad \forall~i\in I \tag{cf.~\ref*{example:BNE:generator:incentive}'} 
	\end{align*}
	
	Incidentally, this case yields identical results as the \emph{standard approach}, in which the power market model is solved according to the two-stage procedure following \cite{ONeill_2005_nonconvexities}, i.e., pay-offs are calculated based on prices from the dual to the energy-balance constraint from the linearized problem, and non-confiscatory make-whole payments are computed ex-post. However, the observation that the two-stage procedure and the integrated multi-objective yield an identical result is specific to this stylized example, and not a general property of the multi-objective program under a no-loss rule.
	
	\item[\textbf{No-loss \& active:}] We add a constraint to the previous case \emph{No-loss rule} stating that only active generators can receive compensation.
	\begin{align*}
	\zeta_{i}&\leq \sum_{t \in T} x_{ti}\, \widetilde{K} \quad \forall~i\in I \tag{cf. \ref*{example:BNE:generator:incentive}''}
	\end{align*}
	
	The rationale for this rule may be to prevent gaming, or in response to the public perception that there should not be payments for \emph{not} providing a service.
\end{itemize}

\begin{table}[b]
	\begin{center}
		\subfloat[Cost for generation and operation status change, initial operational status]{
			\begin{small}
				\begin{tabular}[b]{l|ccc|c}
					& 	$c^G_i$ &	$c\on_i$ &	$c\off_i$ & 
					$x^{init}$ \\ 
					& \scriptsize\$/MWh
					& \scriptsize\$ & \scriptsize\$ \\
					\hline \hline
					$g1$ &	$24$ & 	$100$ &	$500$ & $0$ \\
					$g2$ &	$22$ & 	$140$ &	$350$ & $0$ \\
					$g3$ &	$20$ &	$180$ &	$300$ & $1$ \\
					$g4$ &	$18$ &	$220$ &	$250$ & $1$ \\
					$g5$ &	$16$ &	$250$ &	$220$ & $1$ \\
					$g6$ &	$14$ &	$300$ &	$180$ & $1$ \\
					$g7$ &	$12$ &	$350$ &	$140$ & $0$ \\
					$g8$ &	$10$ &	$500$ &	$100$ & $0$ \\
					$g9$ &	$14$ &	$105$ &	$100$ & $0$
				\end{tabular}
		\end{small} }
		\quad
		\subfloat[b][Utility and maximum demand by load unit and time period]{\begin{small}
				\begin{tabular}[b]{l|cc|cc}
					& 	$u^D_{1j}$ & $u^D_{2j}$ & $d^{max}_{1j}$ & $d^{max}_{2j}$ \\ 
					& \scriptsize\$/MWh & \scriptsize\$/MWh 
					& \scriptsize MW & \scriptsize MW \\
					\hline \hline
					$d1$ &	$25$ & 	$20$ &	$100$ &	$50$ \\
					$d2$ &	$26$ & 	$20$ &	$100$ &	$50$ \\
					$d3$ &	$26$ &	$21$ &	$100$ &	$50$ \\
					$d4$ &	$17$ &	$21$ &	$100$ &	$50$
				\end{tabular}
		\end{small} }
	\end{center}
	\caption{Data for generators and load} \label{table:data}
\end{table}

Table~\ref{table:rent} summarizes the rents earned by the generators in the three market rule cases, as well as the required compensation. Table~\ref{table:iso} shows the resulting nodal prices as well as the amount generated and consumed by each unit. Generators~ $g1$ and~$g2$ are not operational in any of these cases; therefore, they are omitted from the tables wherever all entries are zero.

\begin{table}[t]
	\begin{small}
		\begin{center}
			\begin{tabular}{lc|c@{~}R{0.9 cm}@{~}R{0.7 cm}|c@{~}R{0.9 cm}@{~}R{0.7 cm}|c@{~}R{0.9 cm}@{~}R{0.7 cm}}
				\multicolumn{2}{c|}{} &
				\multicolumn{3}{c|}{No-loss rule} &
				\multicolumn{3}{c|}{Binary equilibrium} &
				\multicolumn{3}{c}{No-loss \& active} \\
				& $x^{init}_i$ & $(x_{ti})$ & $\pi_i$ & $\zeta_i$ &
				$(x_{ti})$ & $\pi_i$ & $\zeta_i$ &
				$(x_{ti})$  & $\pi_i$ & $\zeta_i$ \\
				\hline \hline
				$g3$ & $1$ & $(0,0)$ & $-300$ & $300$ & $(0,0)$ & $-300$ & $15$ & $(1,1)$ & $-525$ & $525$ \\
				$g4$ & $1$ & $(1,1)$ & $-160$ & $160$ & $(0,0)$ & $-250$ & $65$ & $(1,0)$ & $-335$ & $335$ \\
				$g5$ & $1$ & $(1,1)$ & $50$ & & $(1,1)$ & $-50$ & & $(1,1)$ & $-50$ & $50$ \\
				$g6$ & $1$ & $(1,1)$ & $200$ & & $(1,1)$ & $100$ & & $(1,1)$ & $100$ \\
				$g7$ & $0$ & $(1,1)$ & $690$ & & $(1,1)$ & $715$ & & $(1,1)$ & $690$ \\
				$g8$ & $0$ & $(1,1)$ & $680$ & & $(1,1)$ & $730$ & & $(1,1)$ & $550$ \\
				$g9$ & $0$ & $(0,0)$ & 			  & & $(1,1)$ & $-5$ & $5$  & $(0,0)$ &  \\
				\hline \hline
				\multicolumn{2}{l|}{Generator profit} & & $1160$ & $460$ &  & $940$ & $85$ & & $435$ & $910$ \\
				\multicolumn{2}{l|}{Consumer surplus} & & $1380$ & & & $1480$ & & & $1830$ \\
				\multicolumn{2}{l|}{Congestion rent} & & $560$ & & & $640$ & & & $740$ \\
				\hline \hline
				\multicolumn{2}{l|}{Total welfare} & & $3100$ & $460$ & & $3060$ & $85$ & & $3005$ & $910$ \\
			\end{tabular}
		\end{center}
	\end{small}
	\caption{Profits by generator ($\pi_i$) before compensation is disbursed, and rents by stakeholder group for each case; on-off status and initial status by generator} \label{table:rent}
\end{table}

\begin{table}[t]
	\begin{small}
		\begin{center}
			\begin{tabular}{l@{~}c|R{1.3 cm}R{0.8 cm}|R{1.3 cm}R{0.8 cm}|R{1.3 cm}R{0.8 cm}}
				\multicolumn{2}{c|}{ } &
				\multicolumn{2}{c|}{No-loss rule} &
				\multicolumn{2}{c|}{Binary equilibrium} &
				\multicolumn{2}{c}{No-loss \& active} \\
				& &	t1 \tab &	t2  \tab &	t1 \tab &	t2 \tab &	t1 \tab &	t2 \tab \\
				\hline \hline
				Price &	n1 &	18 \tab &	12.8 &	16.5 &	12.8 &	13.5 &	12.8 \\
				&	$n2$ &	18 \tab &	11.6 &	17 \tab &	11.6 &	11 \tab &	11.6 \\
				&	$n3$ &	18 \tab &	14 \tab &	16 \tab &	14 \tab &	16 \tab &	14 \tab \\
				&	$n4$ &	26 \tab&	20 \tab &	26 \tab &	20 \tab&	28.5 &	20 \tab \\
				&	$n5$ &	26 \tab &	18.8 &	26 \tab &	18 \tab &	26 \tab &	18.8 \\
				&	$n6$ &	26 \tab &	17.6 &	27 \tab &	17 \tab &	23.5 &	17.6 \\
				\hline \hline
				Generation 
				&	$g3$ &		 		&	 	  		&				&				& 25 \tab &	25 \tab \\
				&	$g4$ &	40 \tab &	25 \tab &				&				 &	25 \tab &	 \\
				&	$g5$ &	50 \tab &	25 \tab &	40 \tab &	25 \tab &	37.5 &	25 \tab\\
				&	$g6$ &	50 \tab &	30 \tab &	50 \tab &	25 \tab &	50 \tab &	30 \tab \\
				&	$g7$ &	50 \tab &	50 \tab &	50 \tab &	50 \tab &	50 \tab &	50 \tab \\
				&	$g8$ &	50 \tab &	50 \tab &	50 \tab &	50 \tab &	50 \tab &	50 \tab \\
				&	$g9$ &				&				 &	50 \tab & 	40 \tab & 				& \\
				\hline \hline
				Load 
				&	$d1$ &	100 \tab &	50 \tab &	100 \tab &	50 \tab &	100 \tab &	50 \tab \\
				&	$d2$ &	10 \tab &	30 \tab &	65 \tab &	40 \tab &	&	30 \tab \\
				&	$d3$ &	30 \tab &	50 \tab &	 	\tab &	50 \tab &	37.5 &	50 \tab \\
				&	$d4$ &	100 \tab &	50 \tab &	75 \tab &	50 \tab &	100 \tab &	50 \tab \\
				\hline \hline
				Total dispatch & &	240 \tab &	180 \tab &	240 \tab &	190 \tab &	237.5 &	180 \tab
			\end{tabular}
		\end{center}
	\end{small}
	\caption{Solution by market setup (prices in \$/MWh, dispatch in MWh)} \label{table:iso}
\end{table}

\begin{sidewaystable}
	\begin{small}
		\begin{center}
			\begin{tabular}{l|rrrrr|rrrrr|rrrrr}
				&
				\multicolumn{5}{c|}{No-loss rule} &
				\multicolumn{5}{c|}{Binary equilibrium} &
				\multicolumn{5}{c}{No-loss \& active} \\
				& $(0,0)$ &	$(1,0)$ & 	$(0,1)$ &	$(1,1)$ &	$\zeta_i$ & $(0,0)$ &	$(1,0)$ & 	$(0,1)$ &	$(1,1)$ &	$\zeta_i$ & $(0,0)$ &	$(1,0)$ & 	$(0,1)$ &	$(1,1)$ &	$\zeta_i$ \\
				\hline \hline
				$g1$ &	$\mathbf{0}$	&	$-750$ &	$-380$ &	$-530$ & 	&	 $\mathbf{0}$ &	$-787$ &	$-380$ &	$-567$ &	 &	$\mathbf{0}$ &	$-862$ &	$-380$ &	$-642$ & \\
				$g2$ & $\mathbf{0}$		&	$-590$ &	$-370$ &	$-470$ &	&	$\mathbf{0}$ &	$-627$ &	$-370$ &	$-507$ &	 &	 $\mathbf{0}$ &	$-702$ &	$-370$ &	$-582$ & \\
				$g3$ &	$\mathbf{-300}$ &	$-350$ &	$-690$ &	$-260$ &	$300$ &	$\mathbf{-300}$ &	$-375$ &	$-690$ &	$-285$ &	$15$ &	$-300$ &	$-525$ &	$-690$  &	$\mathbf{-435}$ & $435$ \\
				$g4$ &	$-250$ &	$-250$ &	$-630$ &	$\mathbf{-160}$ &	$160$ &	$\mathbf{-250}$ &	$-275$ &	$-630$ &	$-185$ &	$65$ &	$-250$ &	$\mathbf{-425}$ &	$-630$ &	$-335$ &	$425$ \\
				$g5$ &	$-220$ &	$-120$ &	$-520$ &	$\mathbf{50}$ &	 &	$-220$ &	$-220$ &	$-520$ &	$\mathbf{-50}$ &	 & $-220$ &	$-220$ &	$-520$ &	$\mathbf{-50}$ &	$50$ \\
				$g6$ &	$-180$ &	$20$ &	$-480$ &	$\mathbf{200}$ &	 & $-180$ &	$-80$ &	$-480$ &	$\mathbf{100}$ &	 &	$-180$ &	$-80$ &	$-480$ &	$\mathbf{100}$ &	 \\
				$g7$ &	 &	$210$ &	$-10$ &	$\mathbf{690}$ &	 &	 &	$235$ &	$-10$ &	$\mathbf{715}$ &	 &	 &	$210$ &	$-10$ &	$\mathbf{690}$ &	 \\
				$g8$ &	 &	$200$ &	$-120$ &	$\mathbf{680}$ &	 &	 &	$250$ &	$-120$ &	$\mathbf{730}$ &	 &	 &	$75$ &	$-120$ &	$\mathbf{555}$ &	 \\
				$g9$ &	$\mathbf{0}$ &	$-5$ &	$-105$ &	$95$ &	 &	 &	$-105 $&	$-105$ &	$\mathbf{-5}$ &	$5$ & $\mathbf{0}$ 	 &	$-105$ &	$-105$ & \\
			\end{tabular}
		\end{center}
	\end{small}
	\caption{Profits by generator for each on/off option, and compensation (in \$); market outcome in bold} \label{table:deviation}
\end{sidewaystable}

For reasons of illustration, we first discuss the results for the \emph{No-loss} case; as stated before, this yields the same outcome as the method proposed by \citet{ONeill_2005_nonconvexities} with ex-post compensation in this example. Generator~$g3$ is switched off immediately and therefore incurs shut-down costs of~$\$\,300$, while generator~$g4$ is operating, but the nodal price at node~$n2$ is below her marginal costs. Both of these generators are made whole such that they do not incur losses. The deviation incentives of each player are shown in Table~\ref{table:deviation}; here, one can see that generator~$g9$ has a profitable deviation given market prices in the no-loss case. As a consequence, this generator has an incentive to switch on the power plant in spite of the schedule announced by the ISO, as she would earn positive profits given the prevailing market prices. Hence, if self-scheduling is an option for the generator, this outcome would not be a Nash equilibrium and the solution would not be stable against deviations.

In the \emph{binary-equilibrium} case, the market operator could choose to compensate generator~$g9$ to counteract self-scheduling, and then obtain the same dispatch and nodal prices as in the no-loss case. This would require to compensate generator~$g9$ to the tune of $\$\,95$ and generator~$g3$ with $\$\,40$. Generator~$g4$ does not have a profitable deviation, because her losses are at least as great under all alternative dispatch schedules; hence, she does not receive compensation. The objective value of the market operator would be $\$\,2965$ (total welfare of~$\$\,3100$ less $\$\,135$ disbursed as compensation).

However, because of the integrated consideration of market efficiency and compensation payments, the market operator realizes that it is preferable to dispatch generator~$g9$ and instead shut down generator~$g4$, realizing an objective value of \$$\,2975$. Generator~$g9$ now incurs losses, because the resulting locational marginal prices given the new dispatch are lower than her marginal costs; for these losses, she is compensated by the market operator to prevent her from leaving the market. Overall, market efficiency is slightly reduced, but the compensation required to maintain incentive compatibility in this outcome is significantly lower than the payments necessary to guarantee incentive compatibility of the welfare-optimal solution.

The last case, \emph{No-loss \& active}, illustrates how strict market rules can hamper efficient market operation, even when they are intended to mitigate strategic behavior or increase public acceptance of compensation payments. Generator~$g3$ would incur losses from shutting down at the beginning of the period, but cannot receive compensation if she doesn't generate at least in one period; therefore, the market operator dispatches this plant throughout the model horizon. Because this would result in infeasible flows on the network in the second period, the market operator shuts down generator~$g4$ at the end of the first period. Now, this generator incurs losses from generating at a nodal price below marginal costs in the first period and the shut-down costs, and is compensated accordingly.

In this case, welfare is reduced by $3\,$\% compared to the \emph{No-loss} case, while the required compensation payments are almost twice as high. Assuming that every dollar disbursed in compensation payments is assumed to be a $100\,$\% loss, the overall welfare in the system is $30\,$\% lower than in the best-possible outcome. 

\subsection{Numerical implementation and a note on computation} \label{sec:example:computation}
Reformulating the complementarity conditions of the demand-side player~(equations~\ref{example:ISO:KKT}) and the generators (equations~\ref{example:generator:KKT:min} and~\ref{example:generator:KKT:max}) using disjunctive constraints yields a mixed-integer linear program \citep{Fortuny-Amat_1981}. This approach to determine a binary quasi-equilibrium requires $3 \, |T| \; |I| = 54$~binary variables for the generators and $2 \, |T| \, \big( |J| + |L| + |N| -1\big)=68$~binary variables for the disjunctive-constraints reformulation of the ISO. The total number of binary variables is therefore~$|T| \, \big( 2 \, (|I| + |J| + |L| + |N| -1) + |I| \big) =122$. The large scalars~$\widetilde{K}$ for the disjunctive constraints reformulation and the constraints on assigning the correct values~\smash{$\kappa_{ti}\on,\kappa_{ti}\off$} (equations~\ref{example:BNE:kappa_on} and~\ref{example:BNE:kappa_off}) were set to~$1000$.

This compares to $|T| \; |I| = 18$~binary variables to compute the welfare-optimal dispatch and integer pricing following the approach proposed by \cite{ONeill_2005_nonconvexities}. Our method is therefore more computationally expensive, but the number of binary variables increases only linearly in the number of time periods, and sublinear in the number of generators, load units, nodes and lines. Instead solving the resulting equilibrium problem for every permutation of binary variables and checking for profitable deviations ex-post (i.e., the brute-force approach) grows exponentially in complexity and requires solving \mbox{$2^{|T| \; |I|}> 262,000$~(linear)} problems.

The numerical model presented in this section is implemented in GAMS and solved using the GUROBI solver. The code includes the possibility of multiple bidding blocks for each generator and demand unit, as in the model formulated by~\cite{Gabriel_2013_DCMLCP}, even though this option is not used here. The GAMS code is published under a \emph{Creative Commons Attribution 4.0 International License} and is available for download at~\url{https://github.com/danielhuppmann/binary_equilibrium}. An algorithm for enumerating all permutations and checking for deviation incentives ex-post was also implemented to verify the accuracy of our methodology.

\section{Conclusions and outlook} \label{sec:conclusion}
Non-cooperative games with binary decision variables are often encountered in real-world applications, from engineering to economics. It is well known that equilibria in such problems do not necessarily exist, and even if they do exist, finding them is mathematically challenging. The most frequently studied example is the power market uplift problem, which seeks to reconcile the difficulty of finding market-clearing prices -- based on the short-term, efficient dispatch -- with obtaining incentive-compatible outcomes in decentralized, non-cooperative markets. To date, no approach to exactly solve such games exists. 

In this work, we propose an exact solution method for binary equilibrium problems based on computing optimal responses for each player for both values of the binary variable (or vector), rather than assuming a continuous relaxation of binary variables or relaxing the optimality conditions. We then add an explicit incentive-compatibility constraint to ensure that no player has a profitable deviation. We define the notion of a binary quasi-equilibrium to describe situations where no equilibrium exists for the original problem, but in which compensation payments can align the incentives of players and a stable equilibrium is realized.

To this end, we introduce a market operator that acts as an equilibrium selection mechanism according to an upper-level objective function. By recasting the binary equilibrium problem as a hierarchical multi-objective program subject to a binary quasi-equilibrium, our method allows to explicitly incorporate the trade-off between market efficiency and the budget required for compensation payments to obtain an incentive-compatible market outcome. With regard to the power market, this can be interpreted as striking a balance between maximizing short-run welfare (consumer utility less generation costs) and the amount of uplift payments, which are usually funded through taxation, usage fees, or price mark-ups. 

Instead of using the shadow price from a marginal relaxation of the integral constraint, our method yields a ``switch value,'' which is the loss a player would incur if she were to deviate from her individually optimal decision. The switch value can also be readily interpreted as the compensation payment a player should receive if the market operator requires her to deviate from the individually optimal strategy. The switch value can be used for algorithmic improvements and new approaches to solve binary problems (e.g., act as a stopping criterion, guide a branch-and-bound algorithm).

Most importantly, we show that this method can be reformulated and solved as a mixed-binary linear program under general conditions, and that the solution is at least as good as the current practice under certain assumptions. Hence, the approach can be applied to a wide range of real-world problems, including Nash equilibria in energy and natural resource markets. The approach allows to include a variety of market regulations, such as ``no-loss'' rules common in power markets. These rules can be formulated as linear constraints and therefore do not substantially increase the numerical complexity of obtaining a binary quasi-equilibrium. 

The solution method for binary equilibrium problems proposed in this work can be extended to include Generalized Nash games \citep[cf.][]{Harker_1991} or games with individual joint constraints \citep{Nabetani_2011_GNE}, as well as equilibrium problems in discrete rather than binary variables. Furthermore, more general non-cooperative games can be solved in this framework, such as games based on conjectural variations \citep{Wogrin_2013_math_programming}.  Extending the approach to stochastic applications is also straightforward.

One caveat of the proposed method is that the reformulation as a binary optimization program with linear constraints is not directly applicable to market settings where players are aware of their own impact on the final demand price (i.e., Nash-Cournot equilibrium models). The properties of this method in larger-scale applications also requires further investigation to better understand computation time scaling and numerical characteristics.

The market rules obviously differ across various fields. A topic that has received far less attention from the applied Operations Research community is agriculture. Yet it is also a form of binary game between independent farmers (or agro-commercial enterprises), which decide whether to plant a field or not. Interestingly, and in contrast to power markets, there are subsidies being paid in many countries to farmers that let some of their land lay idle. Hence, there is a payment for doing nothing! The key difference may be that farming is seen as a market less prone to exertion of strategic behavior, and that letting a field lay idle is still seen to provide a benefit beyond mere price support, by maintaining the landscape or providing breeding grounds for wildlife -- public acceptance of support schemes may be higher in these cases.

In economic applications, the binary variables can be interpreted as on/off decisions, or as market-entrance or investment decisions in a dynamic, two-stage setting. Lumpy investment in the European power grid capacity by national regulators, each seeking to shift rents towards their domestic constituents, is a natural next application of the binary equilibrium method \citep[cf.][]{Huppmann_2015_natstrat_investment}. The European regulatory agency ACER and the inter-TSO compensation mechanism under its purview are very similar to the structure of the proposed method, where an upper-level coordinator guides non-cooperative players with binary investment decisions.

Because most current power markets are based on a welfare-optimal dispatch with ex-post compensation payments, there may exist numerous gaming opportunities for large market participants in the current setting; the settlement in 2013 between the Federal Energy Regulatory Commission (FERC) and JP Morgan regarding manipulative bidding strategies is a case in point\footnote{See FERC's news release from July 30, 2013, Docket Nos.~IN11-8-000 and~IN13-5-000.\\ \url{http://www.ferc.gov/media/news-releases/2013/2013-3/07-30-13.asp}}. It will be the subject of future research to analyze whether our solution method for binary equilibrium problems will be more or less prone to market power exertion and strategic behavior. Our method explicitly incorporates game-theoretic considerations (i.e., deviation incentives) rather than simplistic no-loss rules, and it includes the trade-off between maximizing market efficiency and minimizing compensation payments. As a consequence, our approach has significant potential to improve the current practice in power market operation. In particular, it allows to include the incentives of non-active players to enter the market; hence, the behavior of non-dispatched generators entering the market through self-commitment (or ``self-scheduling'') can be more effectively addressed.

As a numerical application, we solve a stylized power market uplift problem. We illustrate that the current practice in power market operation can lead to situations where players have profitable deviations. In particular, when considering the nature of the non-cooperative game in our stylized example, the market operator prefers to deviate from the welfare-optimal dispatch, because a slight reduction in market efficiency is traded for a strong decrease of compensation payments. To illustrate the flexibility of our approach, we also solve the model under a hypothetical market regulation stating that a)~no generator may lose money, and b)~only active generators may receive compensation. This yields a welfare loss of $3\,$\% relative to the optimal solution. At the same time, compensation payments are almost twice as high as in the currently used approach, so that compensation payments eat up almost a third of total welfare in the market. We take this as a warning that market rules may have rather counter-intuitive effects, even when they are implemented with the aim of preventing strategic behavior or mitigating other inefficiencies.

%% file: mobqe_appendix.tex
\newpage
\section*{Appendix\\ GAMS codes, a numerical test case \& additional examples}
The GAMS code for the binary-equilibrium model presented in Chapter~\ref{sec:example} is published under a \emph{Creative Commons Attribution 4.0 International License} and is available for download at \url{https://github.com/danielhuppmann/binary_equilibrium}. The repository also includes an application of this method to a natural gas market investment and production game to illustrate how the binary equilibrium approach can be applied to more general cases where players interact in settings with multiple binary variables.

\subsection*{A numerical test case}
To illustrate the numerical properties of the binary-equilibrium method, we apply the power market model presented in Chapter~\ref{sec:example} to the open-source model and data set provided by \citeauthor{Pandzic_Online_REAL} at the Renewable Energy Analysis Lab (REAL), University of Washington, Seattle, USA\footnote{Available at \url{http://www.ee.washington.edu/research/real/gams_code.html}}). This open-source data and model were specifically adapted to provide an easily scalable test case, using demand and wind timeseries over multiple days and apply different settings to the dataset (differentiated by favourable or unfavourable wind situations, wind penetration, power plant cost characteristics, grid congestion, etc.).

The test cases were run on a Dell Optiplex 9020 workstation (Windows 7, 8-core 64bit Intel i7-4790 CPU@3.6GHz, 16GB RAM) using the GUROBI solver in GAMS 24.4.6, and setting the GUROBI options to 3 threads on 6 cores.

For the present purpose, we use 7 consecutive days, each day represented by 12 two-hour timesteps, 48 nodes (two out of three regions in the dataset), 64 power plants, and 79 lines (with transmission constraints). Then, we apply two different scenarios: first, a case with favourable wind conditions, in which not much congestion occurs in the network. In this case, the welfare-optimal solution (768 binary variables) was found within 0.5-2 seconds, except for \emph{day6}, where it took 92 seconds (optimality threshold 0.01\%). Using the respective starting point, GUROBI took between 102 and 150 seconds to determine that this was indeed the optimal outcome, and that no welfare improvement (considering the trade-off between maximum welfare and compensation payments) is possible (9,552 binary variables, at a 0.1\% optimality threshold). For details, see the parameter “report\_summary” in the results file \texttt{output/report\_MOPBQE\_region\_1\_2\_wind\_fav.gdx}.

As a second case, we set the switch in the original dataset to “unfavourable wind conditions”. As before, computation times are up to ten seconds for the social-welfare problem (\emph{WF-max}), and between 102 and 150 seconds for the binary equilibrium problem (\emph{Bin Eq}), except for Day 3, where the solver encountered numerical problems and did not find a solution. 

Over the entire week, congestion rents are much higher than in the “favourable wind case”, and the binary equilibrium method identified reductions of compensation payments of up to 35\%, at no discernible loss of aggregate welfare (i.e. before compensation payments). In Table~\ref{table:numerical_tests}, we show both the compensation required to guarantee no losses as well as the incentive-compatible compensation payments (“IC comp”), in line with the discussion in Section~\ref{sec:model:comparison} (Problem~\ref{twostage:problem}). The compensation payments are around 10\% of generators profits – they may seem small compared to the overall welfare, but this is because the consumer surplus is computed as the difference between the actual price and the load-shedding price bound. Congestion rents are much higher than in the case with favourable wind conditions, and the binary equilibrium method identified reductions of compensation payments of up to 35\%, at no discernible loss of aggregate welfare. For further details, please refer to \texttt{output/report\_MOPBQE\_region\_1\_2\_wind\_unfav.gdx} in the GitHub repository. 

Last, let us point out that there are almost 10,000~dispatch options for the 64~generators, already accounting for eliminating all dispatch schedules that are not possible due to minimum up- or downtime constraints (this is shown in the report “dispatch options” in the gdx report file). If one were to try brute-force enumeration of all permutations of options to determine stable Nash equilibria, this would require to solve $10^{128}$~equilibrium problems for these test instances. In all cases, we also tried to solve the binary-equilibrium model directly without using the starting point derived from the socially optimal outcome. In no case was a solution found within the specified time limit of one hour.

\begin{sidewaystable}
	\begin{small}
		\begin{center}
			\begin{tabular}{ll|rrrrr|rr|rr|r}
				&       & \multicolumn{1}{c}{Generator} & \multicolumn{1}{c}{Wind} & \multicolumn{1}{c}{Consumer} & \multicolumn{1}{c}{Congestion} & \multicolumn{1}{c|}{Gross} & \multicolumn{1}{c}{No-loss} & \multicolumn{1}{c|}{Net} & \multicolumn{1}{c}{IC} & \multicolumn{1}{c|}{Net} & \multicolumn{1}{c}{Time} \\
				&       & \multicolumn{1}{c}{profit} & \multicolumn{1}{c}{rent} & \multicolumn{1}{c}{surplus} & \multicolumn{1}{c}{rent} & \multicolumn{1}{c|}{welfare} & \multicolumn{1}{c}{comp} & \multicolumn{1}{c|}{welfare} & \multicolumn{1}{c}{comp} & \multicolumn{1}{c|}{welfare} & \multicolumn{1}{c}{(sec)} \\
				\hline \hline
				Day 1 & WF-max & 0.092 & 0.277 & 11.677 & 0.026 & 12.071 & 0.023 & 12.048 & 0.016 & 12.1 & 1.9 \\
				& Bin Eq & 0.111 & 0.207 & 11.615 & 0.139 & 12.071 &       &       & 0.011 & 12.061 & 146.9 \\
				&       &       &       &       &       &       &       &       & \emph{33.54\%} &       &  \\
				\hline
				Day 2 & WF-max & 0.102 & 0.295 & 12.520 & 0.029 & 12.946 & 0.021 & 12.925 & 0.014 & 12.931 & 1.0 \\
				& Bin Eq & 0.120 & 0.223 & 12.475 & 0.128 & 12.946 &       &       & 0.010 & 12.936 & 136.1 \\
				&       &       &       &       &       &       &       &       & \emph{32.64\%} &       &  \\
				\hline
				Day 3 & WF-max & 0.129 & 0.290 & 12.240 & 0.037 & 12.696 & 0.007 & 12.690 & 0.012 & 12.684 & 1.4 \\
				& Bin Eq &       &       &       &       &       &       &       &       &       &  \\
				&       &       &       &       &       &       &       &       &       &       &  \\
				\hline
				Day 4 & WF-max & 0.085 & 0.262 & 12.066 & 0.033 & 12.446 & 0.026 & 12.420 & 0.019 & 12.428 & 1.9 \\
				& Bin Eq & 0.099 & 0.131 & 12.022 & 0.195 & 12.446 &       &       & 0.014 & 12.432 & 115.2 \\
				&       &       &       &       &       &       &       &       & \emph{25.41\%} &       &  \\
				\hline
				Day 5 & WF-max & 0.079 & 0.248 & 11.838 & 0.031 & 12.196 & 0.026 & 12.170 & 0.017 & 12.179 & 1.9 \\
				& Bin Eq & 0.096 & 0.187 & 11.764 & 0.150 & 12.196 &       &       & 0.014 & 12.183 & 102.4 \\
				&       &       &       &       &       &       &       &       & \emph{20.93\%} &       &  \\
				\hline
				Day 6 & WF-max & 0.054 & 0.187 & 9.873 & 0.020 & 10.133 & 0.023 & 10.110 & 0.014 & 10.119 & 9.7 \\
				& Bin Eq & 0.076 & 0.300 & 9.750 & 0.008 & 10.133 &       &       & 0.014 & 10.119 & 115.2 \\
				&       &       &       &       &       &       &       &       & \emph{3.47\%} &       &  \\
				\hline
				Day 7 & WF-max & 0.048 & 0.179 & 9.631 & 0.020 & 9.878 & 0.024 & 9.854 & 0.015 & 9.863 & 7.9 \\
				& Bin Eq & 0.066 & 0.280 & 9.527 & 0.006 & 9.878 &       &       & 0.015 & 9.863 & 102.3 \\
				&       &       &       &       &       &       &       &       & \emph{0.68\%} &       &
			\end{tabular}%
		\end{center}
	\end{small}
	\begin{center}
		\caption{Summary of numerical tests under unfavorable wind conditions (welfare and rents in m\$, computation time in seconds) \\ relative reduction of payments required to guarantee incentive compatibility given in italics} \label{table:numerical_tests}
	\end{center}
\end{sidewaystable}